\documentclass{amsart}
\usepackage{hyperref}
\usepackage{graphicx}

\usepackage{appendix}
\usepackage{amssymb}
\usepackage{amsmath}
\usepackage{amsthm}

\usepackage{subfigure}
\usepackage{array}
\usepackage{ esint }
\usepackage{float}
\usepackage{color}
\usepackage[top=2cm,bottom=2cm,left=2cm,right=2cm]{geometry}

 \numberwithin{equation}{section}

\author{Konstantinos Papafitsoros}
\author{Carola-Bibiane Sch\"onlieb}

\email{k.papafitsoros@maths.cam.ac.uk }
\email{C.B.Schoenlieb@damtp.cam.ac.uk}

\address{ Cambridge Centre for Analysis, Department of Applied Mathematics and Theoretical Physics, University of Cambridge, Wilberforce Road, Cambridge CB3 0WA}
\title{A combined first and second order variational approach for image reconstruction}

\date{\today}

\begin{document}
\maketitle
\begin{abstract}
In this paper we study a variational problem in the space of functions of bounded Hessian. Our model constitutes a straightforward higher-order extension of the well known ROF functional (total variation minimisation) to which we add a non-smooth second order regulariser. It combines convex functions of the total variation and the total variation of the first derivatives. In what follows, we prove existence and uniqueness of minimisers of the combined model and present the numerical solution of the corresponding discretised problem by employing the split Bregman method. The paper is furnished with applications of our model to image denoising, deblurring as well as image inpainting.  The obtained numerical results are compared with results obtained from total generalised variation (TGV), infimal convolution and Euler's elastica, three other state of the art higher-order models. The numerical discussion confirms that the proposed higher-order model competes with models of its kind in avoiding the creation of undesirable artifacts and blocky-like structures in the reconstructed images -- a known disadvantage of the ROF model -- while being simple and efficiently numerically solvable.
\end{abstract}

\section{Introduction}
We consider the following general framework of a combined first and second order non-smooth regularisation procedure. For a given datum $u_{0}\in L^{s}(\Omega)$, $\Omega\subset\mathbb{R}^2$, $s=1,2$, we compute a regularised reconstruction $u$ as a minimiser of a combined first and second order functional $H(u)$. More precisely, we are interested in solving
\begin{eqnarray}\label{functional}
\min_{u}\Big\{H(u)&=&\frac{1}{s}\int_{\Omega}|u_{0}-Tu|^{s}~dx+\alpha\int_{\Omega}f(\nabla u)~dx +\beta\int_{\Omega}g(\nabla^{2}u)~dx\Big\}, 
\end{eqnarray}
for $s\in\{1,2\}$, non-negative regularisation parameters $\alpha,\beta$, convex functions $f:\mathbb{R}^{2}\rightarrow \mathbb{R}^+$, $g:\mathbb{R}^{4}\rightarrow \mathbb{R}^+$  with at most linear growth at infinity, and a suitable linear operator $T$, see Section \ref{relax2order} for details.  The appropriate space for this minimisation is the space of functions of bounded Hessian $BH(\Omega)$ which consists of all functions $u\in W^{1,1}(\Omega)$ such that $\nabla u$ is a function of bounded variation.  The idea of this combination of first and second order dynamics is to regularise with a fairly large weight $\alpha$ in the first order term -- preserving the jumps as good as possible -- and using a not too large weight $\beta$ for the second order term such that artifacts (staircasing) created by the first order regulariser are eliminated without introducing any serious blur in the reconstructed image. We will show that for image denoising, deblurring as well as inpainting the model \eqref{functional} offers solutions whose quality (accessed by an image quality measure) is not far off from the ones produced by some of the currently best higher-order reconstruction methods in the field, e.g., the recently proposed total generalised variation (TGV) model \cite{TGV}. Moreover, the computational effort needed for its numerical solution is not much more than the one needed for solving the standard ROF model \cite{rudin1992nonlinear}. For comparison the numerical solution for TGV regularisation is in general about ten times slower than this, see Table \ref{times} at the end of the paper.\\

In this paper we prove existence and uniqueness of \eqref{functional} for the classical setting of the problem in the space $W^{2,1}(\Omega)$ by means of relaxation. The generality of this result includes both the classical variational formulation in $W^{2,1}$, e.g. for the Huber regularised version of the total variation, as well as the non-smooth norm minimisation setting in $BH(\Omega)$, which constitutes the relaxed version of \eqref{functional}. In the numerical part of the paper a discretised version of \eqref{functional} is minimised by means of an operator splitting technique for the cases $f(x)=|x|$ and $g(x)=|x|$, and its application to image denoising, deblurring and inpainting is discussed. 

The rest of the introduction is structured as follows. In Section \ref{sec:context} we phrase the general inverse problem for image reconstruction, which leads us to non-smooth norm minimisation, e.g. total variation minimisation, and eventually to the introduction of higher-order regularisers within this class. This section is followed by a presentation of state of the art higher-order methods in imaging in Section \ref{secrelate} and a discussion of some models from this group which are closely related to \eqref{functional} in Section \ref{secinftgv}.

\subsection{Context}\label{sec:context}
A general inverse problem in imaging reads as follows. Observing or measuring data $u_{0}\in\mathcal H$ in a suitable Hilbert space of real functions defined on a domain $\Omega$, we seek for the original or reconstructed image $u$ that fulfils the model
 \begin{equation}\label{invprob}
u_0=Tu+n,
\end{equation}
where $T$ is a linear operator in $\mathcal H$, i.e., $T\in\mathcal L(\mathcal H)$, and $n=n(u_0)$ denotes a possible noise component which -- depending on the noise statistics -- might depend on the data $u_0$. The operator $T$ is the forward operator of this problem. Examples are blurring operators (in which case $Tu$ denotes the convolution of $u$ with a blurring kernel), $T=\mathcal F$ the Fourier transform or $T=\mathcal P_n$ a projection operator onto a subspace of $\mathcal H$ (i.e., only a few samples of $u$ are given). 

To reconstruct $u$ one has to invert the operator $T$. This is not always possible since in many applications a problem can be ill-posed and further complicated by interferences like noise. In this case a common procedure in inverse problems is to add a-priori information to the model, which in general is given by a certain regularity assumption on the image $u$. Hence, instead of solving \eqref{invprob} one computes $u$ as a minimiser of
$$
\mathcal J(u) =  \Phi\left(u_{0},Tu\right)  +\alpha \psi(u),
$$
defined in a suitable Banach space $\mathcal H_\psi$. Here, $\psi$ models the regularity assumption on $u$ with a certain regularity parameter $\alpha>0$ and is called the regulariser of the problem, and $\Phi$ is a distance function in $\mathcal H$ that enforces \eqref{invprob}. The latter depends on the statistics of the data $u_0$, which can be either estimated or are known from the physics behind the acquisition of $u_0$. For $u_{0}$ corrupted by normally distributed additive noise, this distance function is the squared $L^2$ norm of $u_{0}-Tu$. For the choice of the regulariser $\psi$, squared Hilbert space norms have a long tradition in inverse problems. The most prominent example is $H^1$ regularisation
\begin{equation}\label{gaussfilt}
\min_{u\in H^1}\left\{\frac{1}{2} \|u_{0}-Tu\|^2_{L^2(\Omega)} +\alpha \|\nabla u\|_{L^2(\Omega)}^2  \right\},
\end{equation}
see also \cite{Tikhonov,Whittaker}. For $T=Id$, the gradient flow of the corresponding Euler-Lagrange equation  of \eqref{gaussfilt} reads $u_{t}=\alpha\Delta u -u +u_{0}$. The result of such a regularisation technique is a linearly, i.e., isotropically, smoothed image $u$, for which the smoothing strength does not depend on $u_{0}$. Hence, while eliminating the disruptive noise in the given data $u_{0}$ also prominent structures like edges in the reconstructed image are blurred. This observation gave way to a new class of non-smooth norm regularisers, which aim to eliminate noise and smooth the image in homogeneous areas, while preserving the relevant structures such as object boundaries and edges. More precisely, instead of \eqref{gaussfilt} one considers the following functional over the space $W^{1,1}(\Omega)$:
\begin{equation}\label{1storder}
\mathcal J(u) =  \frac{1}{2} \|u_{0}-Tu\|_{L^2(\Omega)}^2+\int_{\Omega} f(\nabla u) ~ dx,
\end{equation}
where $f$ is a function from $\mathbb{R}^{2}$ to $\mathbb{R}^{+}$ with at most linear growth, see \cite{vese2001study}. As stated in \eqref{1storder} the minimisation of $\mathcal J$ over $W^{1,1}(\Omega)$ is not well-posed in general. For this reason relaxation procedures are applied, which embed the optimisation for $\mathcal J$ into the optimisation for its lower semicontinuous envelope within the larger space of functions of bounded variation, see Section \ref{preliminaries}. The most famous example in image processing is $f(x)=|x|$, which for $T=Id$ results in the so-called ROF model \cite{rudin1992nonlinear}. In this case the relaxed formulation of \eqref{1storder} is the total variation denoising model, where $\|\nabla u\|_{L^1(\Omega)}$ is replaced by the total variation $|D u|(\Omega)$ and $\mathcal J$ is minimised over the space of functions of bounded variation. Other examples for $f$ are regularised versions of the total variation like $f(x)=\sqrt{x^2+\epsilon^2}$ for a positive $\epsilon\ll 1$ \cite{Acar94analysisof,FengProhl}, the Huber-regulariser and alike \cite{charbonnier:two,Green90bayesianreconstructions,ChambolleLions,Nikolova}.  The consideration of such regularised versions of $|\nabla u|$ is sometimes of advantage in applications where perfect edges are traded against a certain smoothness in homogeneous parts of the image, \cite{pcbc09}. Moreover such regularisations become necessary for the numerical solution of \eqref{1storder} by means of time-stepping \cite{vese2001study} or multigrid-methods \cite{vogel95,Henn2004,Vassilevski97acomparison} for instance.

As these and many more contributions in the image processing community have proven, this new non-smooth regularisation procedure indeed results in a nonlinear smoothing of the image, smoothing more in homogeneous areas of the image domain and preserving characteristic structures such as edges. In particular, the total variation regulariser is tuned towards the preservation of edges and performs very well if the reconstructed image is piecewise constant. The drawback of such a regularisation procedure becomes apparent as soon as one considers images or signals (in 1D) which do not only consist of flat regions and jumps, but also possess slanted regions, i.e., piecewise linear parts. The artifact introduced by total variation regularisation in this case is called staircasing. Roughly this means that the total variation regularisation of a noisy linear function $u_{0}$ in one dimension is a staircase $u$, whose $L^2$ norm is close to $u_{0}$, see Figure \ref{staircase1D}.  In two dimensions this effect results in blocky-like images, see Figure \ref{staircase}. In one dimension this effect has been rigorously studied in \cite{dal2009higher}.

\begin{figure*}
\begin{center}
\subfigure[Noisy, piecewise linear signal]{\includegraphics[width=8cm]{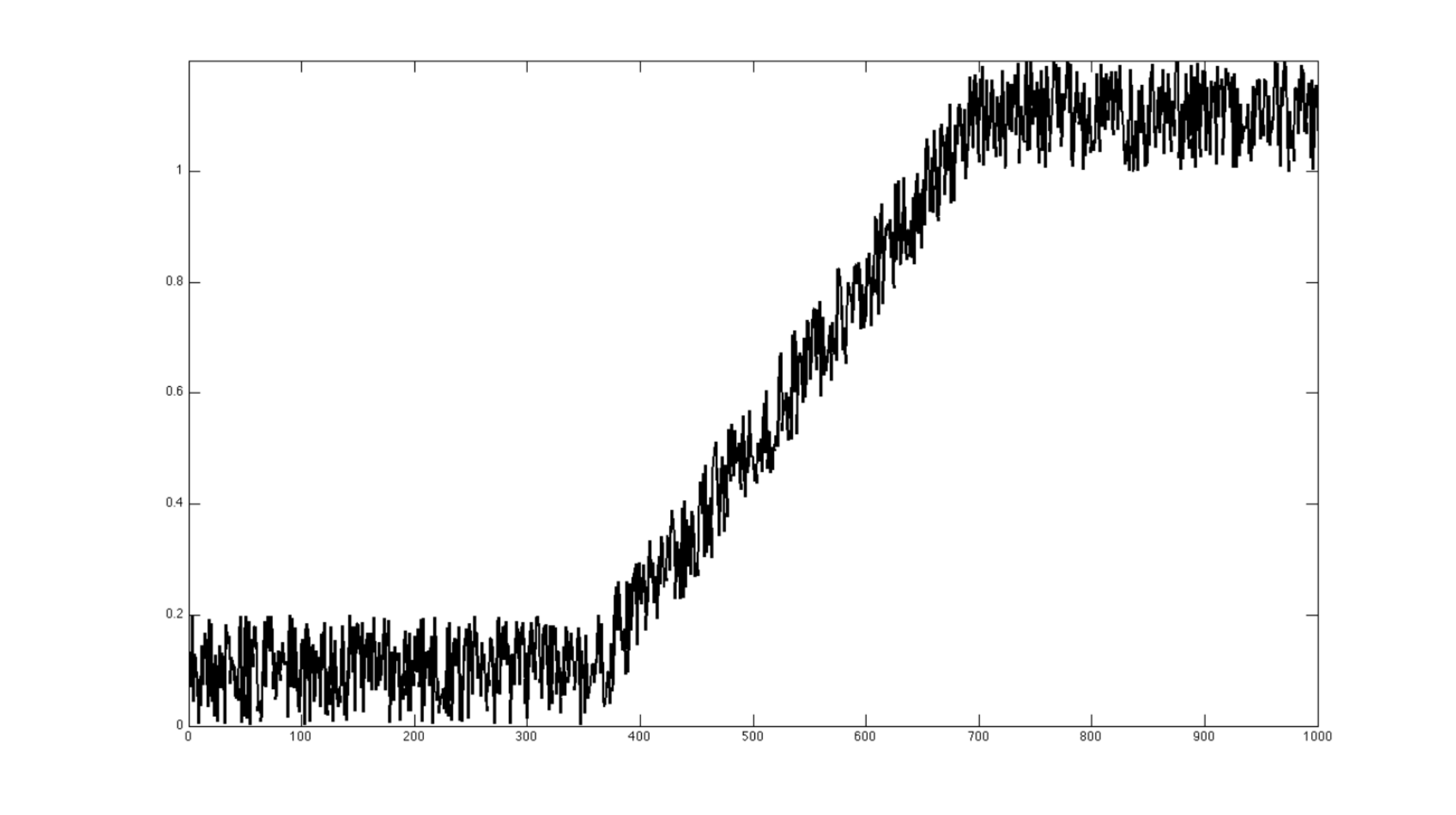}}
\subfigure[First order total variation denoised signal]{\includegraphics[width=8cm]{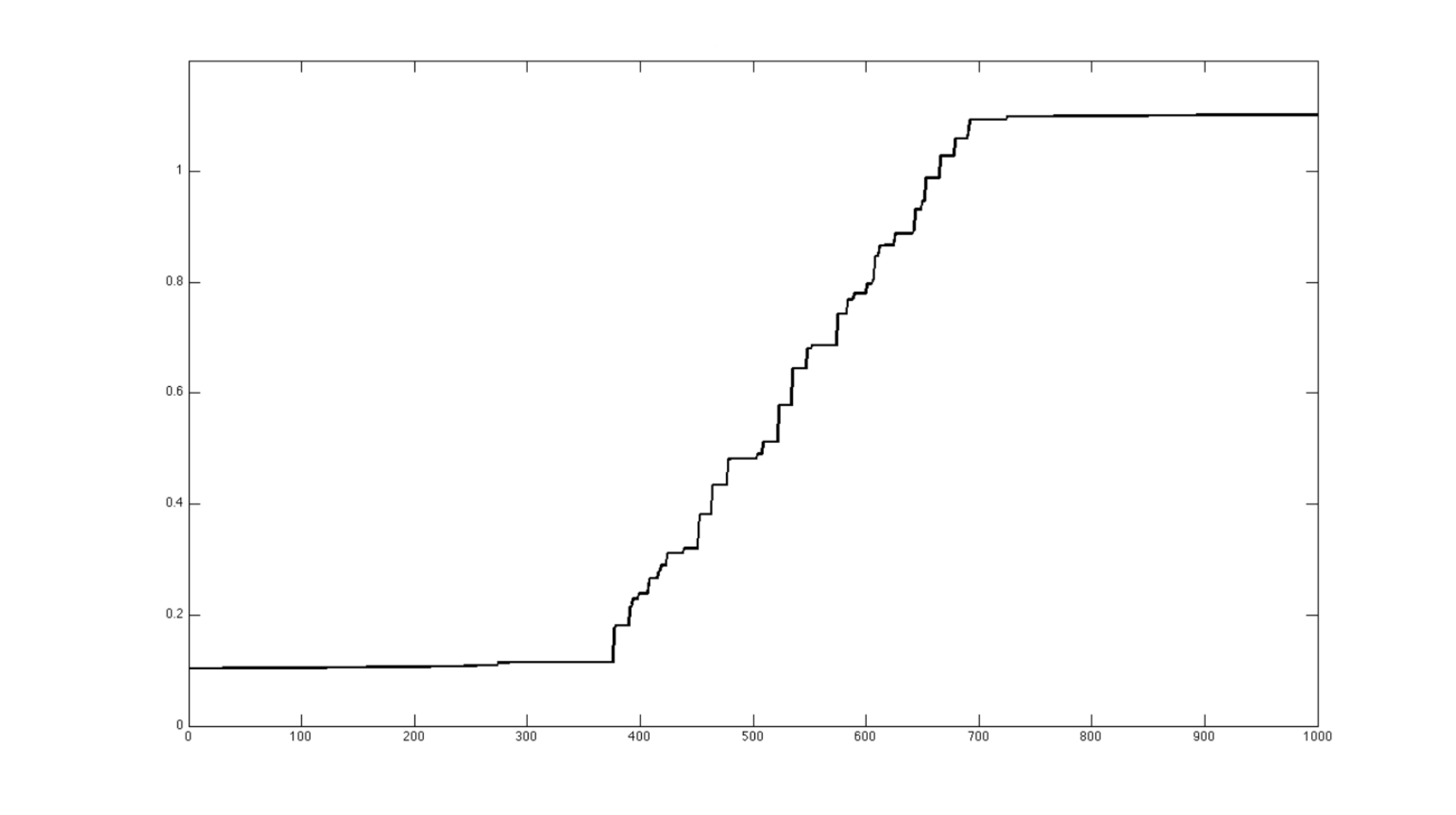}}
\end{center}
\caption{Illustration of the staircasing effect in one space dimension}
\label{staircase1D}
\end{figure*}

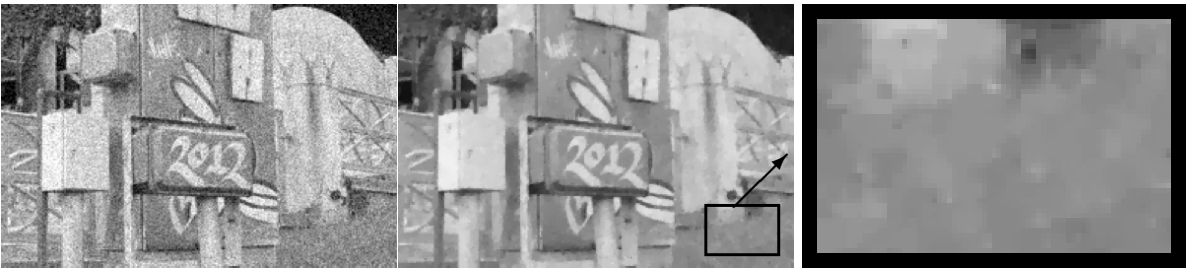
\begin{figure*}
\begin{center}
\includegraphics[height=3.5cm]{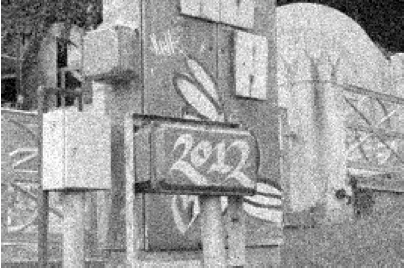}\includegraphics[height=3.5cm]{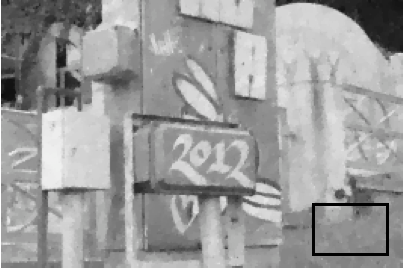} \includegraphics[height=3.5cm]{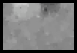}
\begin{picture}(100,0)
\thicklines\put(105,35){\vector(1,1){20}}
\end{picture}
\vspace*{-0.5cm}
\caption{Total variation image denoising and the staircasing effect: (l.) noisy image, (m.) denoised image, (r.) detail of the bottom right hand corner of the denoised image to visualise the staircasing effect (the creation of blocky-like patterns due to the total variation regulariser)}
\label{staircase}
\end{center}
\end{figure*}

 One way to reduce this staircasing effect is in fact to ``round off'' the total variation term by using regularised versions defined by functions $f$ as indicated above, e.g., Huber regularisation \cite{pcbc09}. However, such a procedure can only reduce these artifacts to a certain extent. For instance, the Huber-type regularisation will eliminate the staircasing effect only in areas with small gradient. Another way of improving total variation minimisation is the introduction of higher-order derivatives in the regulariser as in \eqref{functional}. Especially in recent years, higher-order versions of non-smooth image enhancing methods have been considered.

\subsection{Related work}\label{secrelate}
Already in the pioneering paper of Chambolle and Lions \cite{ChambolleLions} the authors propose a higher-order method by means of an inf-convolution of two convex regularisers. Here, a noisy image is decomposed into three parts $u_{0}=u_1+u_2+n$ by solving
\begin{eqnarray}\label{infconv}
\min_{(u_1,u_2)} \Big\{ \frac{1}{2} \|u_{0}-u_1-u_2\|_{L^2(\Omega)}^2 +\alpha \int_\Omega |\nabla u_1|~ dx + \beta \int_\Omega |\nabla^2 u_2|~ dx  \Big\},
\end{eqnarray}
where $\nabla^2 u_2$ is the distributional Hessian of $u_2$. Then, $u_1$ is the piecewise constant part of $u_0$, $u_2$ the piecewise smooth part and $n$ the noise (or texture). Along these lines a modified infimal-convolution approach has been recently proposed in the discrete setting in \cite{SS08,setzer2011infimal}. Another attempt to combine first and second order regularisation originates from Chan, Marquina, and Mulet \cite{chan2001high}, who consider total variation minimisation together with weighted versions of the Laplacian. More precisely, they consider a regularising term of the form
$$
\alpha \int_\Omega |\nabla u|~ dx + \beta \int_\Omega f(|\nabla u|) (\Delta u)^2~ dx,
$$
where $f$ must be a function with certain growth conditions at infinity in order to allow jumps. The well-posedness of the latter in one space dimension has been rigorously analysed by Dal Maso, Fonseca, Leoni and Morini \cite{dal2009higher} via the method of relaxation.

The idea of bounded Hessian regularisers was also considered by Lysaker et al. \cite{LLT03,LT06}, Chan et al. \cite{CEP07}, Scherzer et al. \cite{Sch98a,HS06}, Lai at al. \cite{lairidge}  and Bergounioux and Piffet \cite{Piffet}. In these works the considered model has the general form
$$
\min_u \left\{\frac{1}{2} \|u_{0}-u\|_{L^2(\Omega)}^2+\alpha |\nabla^2 u|(\Omega)   \right\}.
$$
In Lefkimmiatis et al. \cite{lefkimmiatis2010hessian}, the spectral norm of the Hessian matrix is  considered.
Further, in \cite{PS08} minimisers of functionals which are regularised by the total variation of the $(l-1)$st derivative, i.e.,
$$
|D\nabla^{l-1}u|(\Omega),
$$
are studied. Another interesting higher-order total variation model is proposed by Bredies et al. \cite{TGV}. The considered regulariser is called total generalised variation (TGV) and is of the form
\begin{eqnarray}\label{tgv}
\mathrm{TGV}^k_\alpha(u) = \sup\Big\{\int_\Omega u\mathrm{div}^k\xi~ dx:\;  \xi\in C_c^k(\Omega, \mathrm{Sym}^k(\mathbb{R}^d)),\; \|\mathrm{div}^l \xi\|_\infty\leq\alpha_l,\; l=0,\ldots,k-1\Big\},
\end{eqnarray}
where $\mathrm{Sym}^k(\mathbb{R}^d)$ denotes the space of symmetric tensors of order $k$ with arguments in $\mathbb{R}^d$, and $\alpha_l$ are fixed positive parameters. Its formulation for the solution of general inverse problems was given in \cite{tgvcolour,BredValk}.

Properties of higher-order regularisers in the discrete setting in terms of diffusion filters are further studied in \cite{DiWeBu09}. Therein, the authors consider the Euler-Lagrange equations corresponding to minimisers of functionals of the general type
\begin{equation}\label{JJ}
\mathcal{J}(u) = \int_\Omega (u_0-u)^2 ~ dx + \alpha\int_\Omega f\left(\sum_{|\beta|=p} |D^\beta u|^2\right) ~ dx,
\end{equation}
for different non-quadratic penaliser functions $f$. Moreover, Bertozzi and Greer \cite{BG04} have rigorously studied the fourth-order evolution equation which arises as a gradient flow of $\int G(\Delta u)$, where $G$ is a nondecreasing function of quadratic growth in a neighbourhood of $0$ and at most linear growth at infinity. Solutions of this model are called low curvature image simplifiers and are given by
$$
u_t = -\alpha\Delta(\mathrm{arctan (\Delta u)}) + (u_{0}-u),
$$ 
when $G(s)=s\arctan (s)-1/2\log (s^{2}+1)$.

 Higher-order inpainting methods in general perform much better than first order methods -- like total variation inpainting -- because of the additional directional information used for the interpolation process. 
 Euler's elastica is a popular higher-order variational method  \cite{chan2002euler,tai2010fast}.
There, the regularising term  reads:
 \[ \int_{\Omega} \left (\alpha+\beta\left (\nabla \cdot \frac{\nabla u}{|\nabla u|} \right )^{2}  \right )|\nabla u|~dx,\]
 i.e., is a combination of the total variation and the curvature of the level lines of $u$ (a nonlinear second order regularising term).
 Other examples of higher-order inpainting are the Cahn-Hilliard inpainting \cite{BEG}, TV-H$^{-1}$ inpainting \cite{TVH1_1,TVH1_2} and Hessian-based surface restoration \cite{lairidge}.  
\subsection{Relation of our model to TGV, infimal convolution regularisation, higher-order diffusion filters and Euler's elastica}\label{secinftgv}
In this section we want to analyse the connection of our combined first and second order approach \eqref{functional} with infimal convolution \eqref{infconv} \cite{ChambolleLions,SS08,setzer2011infimal}, the total generalised  variation regulariser of order two \eqref{tgv}  \cite{TGV} and with higher-order diffusion filters \cite{DiWeBu09}. Moreover, in the case of inpainting, we discuss the connection of our model to Euler's elastica \cite{chan2002euler,tai2010fast}.

In the case of inf-convolution \eqref{infconv} the regularised image $u=u_1+u_2$ consists of a function $u_1\in BV(\Omega)$ and a function $u_2\in BH(\Omega)$ which are balanced against each other by positive parameters $\alpha,\beta$. Differently, a minimiser $u$ of \eqref{functional} is in $BH(\Omega)$ as a whole and as such is more regular than the infimal convolution minimiser which is a function in $BV(\Omega)$. Hence, infimal convolution reproduces edges in an image as perfect jumps while in our combined first and second order total variation approach edges are lines where the image function has a large but finite gradient everywhere. We believe that our approach \eqref{functional} can be made equivalent to infimal convolution if combined with the correct choice of adaptive regularisation, e.g. \cite{HintermuellerDong,frick2011statistical}. More precisely, we replace the two constant parameters $\alpha$ and $\beta$ by spatially varying functions $\alpha(x),\beta(x)$ and minimise for $u$
$$
\frac{1}{2} \int_\Omega (u_0-u)^2~ dx + \int_\Omega \alpha(x) |\nabla u|~ dx + \int_\Omega \beta(x) |\nabla^2 u| ~dx.
$$    
Then, we can choose $\alpha$ and $\beta$ according to \eqref{infconv}, i.e., $\alpha=0$ where $u=u_2$, $\beta=0$ where $u=u_1$, and $\alpha/\beta$ correctly balancing $u_1$ and $u_2$ in the rest of $\Omega$.  However, let us emphasise once more that this is not our intention here. \\

The relation of \eqref{functional} to the regularisation approach with total generalised variation \cite{TGV} of order $2$ can be understood through its equivalence with the modified infimal convolution approach \cite{setzer2011infimal} in the discrete setting. The total generalised variation of order $2$ is defined for a positive multi-index $\mathbf\alpha=(\alpha_0,\alpha_1)$ as
\begin{eqnarray*}
\mathrm{TGV}_\alpha^2(u)& =& \sup\Big\{\int_\Omega u~\mathrm{div}^2 v~ dx :\; v\in C_c^2(\Omega,\mathrm{Sym}^2(\mathbb{R}^d)), \; \|\mathrm{div}^l v\|_{\infty}\leq \alpha_l, \;l=0,1\Big\},
\end{eqnarray*}
where $\mathrm{Sym}^2(\mathbb{R}^d)$ is the space of symmetric tensors of order $2$ with arguments in $\mathbb{R}^d$. An alternative definition of TGV$_{\alpha}^{2}$ was proven in \cite{BredValk}
\begin{eqnarray*}
\mathrm{TGV}_\alpha^2(u)& =& \min_{w\in \text{BD}(\Omega)}\alpha_{1}|Du-w|(\Omega)+\alpha_{0}|\mathcal{E}w|(\Omega)
\end{eqnarray*}
where BD$(\Omega)$ is the space of functions of bounded deformation, $\mathcal{E}w$ is the distributional symmetrised gradient of $w$ and $|\cdot|(\Omega)$ denotes the total variation measure evaluated on $\Omega$.
\\

The relation to higher-order diffusion filters as analysed in \cite{DiWeBu09} becomes apparent when considering the Euler-Lagrange equation of \eqref{functional} in the case $T=Id$ and $f$, $g$ having the form $f(x)=h(|x|)$, $g(x)=h(|x|)$, where $h$ is convex and has at most linear growth. Namely, with appropriate boundary conditions we obtain the following Euler-Lagrange equation
\begin{eqnarray}\label{eulerlagrange}
u-u_{0} = \alpha\,\mathrm{div}\left(h'(|\nabla u|) \frac{\nabla u}{|\nabla u|}\right) - \beta \mathrm{div}^{2}\left(h'(|\nabla^2 u|) \frac{\nabla^2 u}{|\nabla^2 u|}\right).\nonumber
\end{eqnarray}
This simplifies for the case $h(x)=\sqrt{|x|^2+\epsilon^2}$ to a regularised first-second order total variation reaction-diffusion equation that reads
\begin{eqnarray*}
u-u_{0}= \alpha\,\mathrm{div}\left( \frac{\nabla u}{\sqrt{|\nabla u|^2+\epsilon^2}}\right) - \beta \mathrm{div}^{2} \left(\frac{\nabla^2 u}{\sqrt{|\nabla^2 u|^2 + \epsilon^2}}\right).
\end{eqnarray*}
The consideration of the corresponding evolution equations for \eqref{eulerlagrange} in the presence of different choices of penalising functions $h$ promises to give rise to additional properties of this regularisation technique and is a matter of future research.

As far as inpainting is concerned, we examine here the connection of our method to Euler's elastica. Depending on how each of the terms in the Euler's elastica regulariser are weighted, the interpolation process is performed differently. If a larger weight is put on the total variation the interpolation results into an image with sharp edges, which however can get disconnected if the scale of the gap is larger than the scale of the object whose edges should be propagated into it. This behaviour is a validation of the so-called ``good continuation principle'' defined by the Gestaltist school \cite{diderot} and not desirable in image inpainting. Putting a larger weight on the curvature term however resolves this issue and gives satisfying results with respect to the continuation principle. The right combination (balance) of these two terms seems to result into a good tradeoff between ``sharpness'' and ``continuation'' of edges. However, one disadvantage of the Euler's elastica inpainting model for analytic and numerical issues is that it is a non-convex minimisation problem. In particular, numerical algorithms are in general not guaranteed to converge to a global minimiser, only local minimisation can be achieved. The minimisation of functional \eqref{functional} can be seen as a  convex simplification of the Euler's elastica idea, where we have replaced the non-convex curvature by the convex total variation of the first derivative of $u$.

For more discussion and comparison of higher-order regularisers we recommend Chapters 4.1.5-4.1.7 and 6.4 in \cite{BenningDiss}.\\

\noindent \emph{Outline of the paper:} In Section \ref{preliminaries} we give a brief introduction to Radon measures, convex functions of measures and functions of bounded variation. In Section \ref{relax2order} we introduce the variational problem \eqref{functional} and the space $BH(\Omega)$ that this functional is naturally defined in. We define two topologies on $BH(\Omega)$ and we identify the lower semicontinuous envelope of \eqref{functional} with respect to these topologies. Finally, we prove the well-posedness -- existence, uniqueness, stability -- of the minimisation of the relaxed functional using standard techniques from calculus of variations and Bregman distances. Section \ref{specialsession} deals with two special versions of \eqref{functional}, the anisotropic version and the case with the $L^{1}$ norm in the fidelity term. In Section \ref{numericssection} we introduce the corresponding discretised problem and we propose the split Bregman method for its numerical implementation in the case $f(x)=|x|$, $g(x)=|x|$. In Sections \ref{denoising}, \ref{deblurring} and \ref{inpainting} we present some numerical examples of our method in image denoising, deblurring and inpainting respectively. Finally, in Section \ref{comparison} we discuss how our approach compares with other higher-order methods like infimal convolution (denoising, deblurring), total generalised variation (denoising, deblurring) and Euler's elastica (inpainting).

\section{Preliminaries}\label{preliminaries}
In this section, we introduce some basic notions that we are going to use.  A reader familiar with Radon measures, $BV$ functions and relaxed functionals can quickly go through Sections \ref{radonsection}, \ref{bvsection} and \ref{relaxedsection} respectively.  Section \ref{convexsection} familiarises the reader with convex functions of measures, a perhaps less known subject.\\

\noindent \textbf{Remarks on standard notation}: As usual, we denote with $\mathcal{L}^{n}$ the Lebesgue measure. Different notions are denoted by $|\cdot|$: When it is applied on vectors or matrices it denotes the Euclidean norm (vector) or the Frobenius norm (matrices). When it is applied on measures it denotes the total variation measure while when it is applied on Borel subsets of $\mathbb{R}^{n}$ it denotes the Lebesgue measure of that subset. Finally, $|\cdot|_{1}$ denotes the $\ell_{1}$ norm in $\mathbb{R}^{n}$ and $(\cdot,\cdot)$ denotes the standard Euclidean inner product.
\subsection{Finite Radon measures}\label{radonsection}
All our notation and definitions are consistent with \cite{AmbrosioBV}. From now on, $\Omega$ denotes an open set in $\mathbb{R}^{n}$. We define the space $[\mathcal{M}(\Omega)]^{m}$ to be the space of $\mathbb{R}^{m}$-valued finite Radon measures. The total variation measure of $\mu\in [\mathcal{M}(\Omega)]^{m}$ is denoted by $|\mu|$. We say that a sequence $(\mu_{k})_{k\in \mathbb{N}}$ in  $[\mathcal{M}(\Omega)]^{m}$ converges weakly$^{\ast}$ to a measure $\mu\in [\mathcal{M}(\Omega)]^{m}$ and we write $\mu_{k}\rightharpoonup\mu$ if $\lim_{k\to\infty}\int_{\Omega}u \,d\mu_{k}=\int_{\Omega}u\,d\mu$
 for all $u\in C_{0}(\Omega)$, the completion of $C_{c}(\Omega)$ endowed with the supremum norm. Thus $(\mu_{k})_{k\in\mathbb{N}}$ converges weakly$^{\ast}$ to $\mu$ if it converges weakly$^{\ast}$ component-wise. We will often consider the Lebesgue decomposition of a $\mathbb{R}^{m}$-valued finite Radon measure $\mu$ with respect to a $\sigma$-finite positive Borel measure $\nu$ :
 \[\mu=\mu^{a}+\mu^{s}=\left(\frac{\mu}{\nu}\right)\nu +\mu^{s},\] 
where  $\mu^{a}$ is the absolutely continuous part of $\mu$ with respect to $\nu$, $\mu^{s}$ is the singular part and $(\mu/\nu)$ denotes the density function of $\mu$ with respect to $\nu$ (Radon-Nikod\'ym derivative). Again this is nothing else than the usual Lebesgue decomposition regarded component-wise. Recall also that any $\mu\in[\mathcal{M}(\Omega)]^{m}$ is absolutely continuous with respect to its total variation measure $|\mu|$ and thus we obtain the \emph{polar decomposition} of $\mu$
\[\mu=\left(\frac{\mu}{|\mu|} \right)|\mu|,\quad with\quad \left |\frac{\mu}{|\mu|} \right |=1,\;\;|\mu|\;\; a.e..\]
\subsection{Convex functions of measures}\label{convexsection} Let $g$ be a continuous function from $\mathbb{R}^{m}$ to $\mathbb{R}$ which is positively homogeneous of degree $1$, i.e., for every $x\in \mathbb{R}^{m}$ 
\[g(tx)=tg(x),\quad \forall t\ge 0.\]
Given a measure $\mu\in[\mathcal{M}(\Omega)]^{m}$, we define the $\mathbb{R}$-valued measure $g(\mu)$ as follows:
\[g(\mu):=g\left(\frac{\mu}{|\mu|} \right)|\mu|.\]
It can be proved that if $g$ is a convex function then $g(\cdot)$ is a convex function in $[\mathcal{M}(\Omega)]^{m}$ and if $\nu$ is any positive measure such that $\mu$ is absolutely continuous with respect to $\nu$ then 
\[g(\mu)=g\left(\frac{\mu}{\nu} \right)\nu.\]
We refer the reader to Proposition \ref{convexmeasure} in Appendix \ref{appendixA} for a proof of the above statement.
Suppose now that $g$ is not necessarily positively homogeneous but it is a continuous function from $\mathbb{R}^{m}\to \mathbb{R}$ which is convex and has at most linear growth at infinity, i.e., there exists a positive constant $K$ such that
\[|g(x)|\le K(1+|x|),\quad \forall x\in \mathbb{R}^{m}.\]
In that case the \emph{recession function} $g_{\infty}$ of $g$ is well defined everywhere, where
\[g_{\infty}(x):=\lim_{t\to\infty}\frac{g(tx)}{t},\quad\forall x\in \mathbb{R}^{m}.\]
It can be proved that $g_{\infty}$ is a convex function and positively homogeneous of degree $1$. Given a measure $\mu\in [\mathcal{M}(\Omega)]^{m}$ we consider the Lebesgue decomposition with respect to Lebesgue measure $\mathcal{L}^{n}$, 
$\mu=(\mu/\mathcal{L}^{n})\mathcal{L}^{n}+\mu^{s}$ and we define the $\mathbb{R}$-valued measure $g(\mu)$ as follows:
\begin{eqnarray}\label{defconv}
g(\mu)&=&g\left (\frac{\mu}{\mathcal{L}^{n}} \right )\mathcal{L}^{n}+g_{\infty}(\mu^{s})\nonumber\\&=&g\left (\frac{\mu}{\mathcal{L}^{n}} \right )\mathcal{L}^{n}+g_{\infty}\left(\frac{\mu^{s}}{|\mu^{s}|} \right)|\mu^{s}|.
\end{eqnarray}
We refer the reader to Theorem \ref{BF} in Appendix \ref{appendixA} for a result regarding lower semicontinuity of convex functions of measures with respect to the weak$^{\ast}$ convergence.

\subsection{The space $[BV(\Omega)]^{m}$}\label{bvsection}
We recall that a function $u\in L^{1}(\Omega)$ is said to be a function of bounded variation or else $u\in BV(\Omega)$ if its distributional derivative can be represented by a $\mathbb{R}^{n}$-valued finite Radon measure, which is denoted by $Du$. This means that
\[\int_{\Omega}u\,\partial_{i}\phi~ dx=-\int_{\Omega}\phi ~d D_{i}u,\;\; \forall \phi\in C_{c}^{1}(\Omega), \;i=1,\ldots,n.\]
for some $\mathbb{R}^{n}$-valued finite Radon measure $Du=(D_{1}u,\ldots,D_{n}u)$. The absolutely continuous part of $Du$ with respect to Lebesgue measure $\mathcal{L}^{n}$ is denoted by $\nabla u$.
It is immediate that $W^{1,1}(\Omega)$$\subseteq$ $BV(\Omega)$ since if $u\in W^{1,1}(\Omega)$ then $Du=\nabla u \mathcal{L}^{n}$. Consistently, we say that a function $u=(u^{1},\ldots,u^{m})\in [L^{1}(\Omega)]^{m}$ belongs to $[BV(\Omega)]^{m}$ if 
\[\int_{\Omega}u^{a}\partial_{i}\phi ~dx=-\int_{\Omega}\phi~ d D_{i}u^{a}, \quad i=1,\ldots,n,\;a=1,\ldots,m. \]
In that case $Du$ is an $m\times n$ matrix-valued measure. A function u belongs to $ [BV(\Omega)]^{m}$ if and only if its variation in $\Omega$, $V(u,\Omega)$ is finite, where,
\begin{eqnarray*}
V(u,\Omega)=\sup\Bigg\{ \sum_{a=1}^{m}\int_{\Omega}u^{a}\text{div}\phi^{a}~dx: \phi\in [C_{c}^{1}(\Omega)]^{mn}, \; \|\phi\|_{\infty}\le1 \Bigg \}.
\end{eqnarray*}
Moreover if $u\in [BV(\Omega)]^{m}$ then $|Du|(\Omega)=V(u,\Omega)$ and if $u\in [W^{1,1}(\Omega)]^{m}$, then $|Du|(\Omega)=\int_{\Omega}|\nabla u|dx$, where $|\nabla u|=\left(\sum_{a=1}^{m}\sum_{i=1}^{n} (\partial_{i}u^{a})^{2}\right )^{1/2}$.  The space $[BV(\Omega)]^{m}$ endowed with the norm $\|u\|_{BV(\Omega)}:=\int_{\Omega}|u|~dx+|Du|(\Omega)$ is a Banach space.
It can be shown that if $Du=0$, then $u$ is equal to a constant a.e. in any connected component of $\Omega$.

Suppose that $(u_{k})_{k\in\mathbb{N}}$, $u$ belong to $[BV(\Omega)]^{m}$. We say that the sequence $(u_{k})_{k\in \mathbb{N}}$ converges to $u$ \emph{weakly}$^{\ast}$ in $[BV(\Omega)]^{m}$ if it converges to $u$ in $[L^{1}(\Omega)]^{m}$ and the sequence of measures $(Du_{k})_{k\in\mathbb{N}}$ converges weakly$^{\ast}$ to the measure $Du$. It is known that $(u_{k})_{k\in \mathbb{N}}$ converges to $u$ weakly$^\ast$ in  $[BV(\Omega)]^{m}$ if and only if $(u_{k})_{k\in \mathbb{N}}$ is bounded in $[BV(\Omega)]^{m}$ and converges to $u$ in $[L^{1}(\Omega)]^{m}$. 
The usefulness of the introduction of the weak$^\ast$ convergence is revealed in the following compactness result: Suppose that the sequence $(u_{k})_{k\in \mathbb{N}}$ is bounded in $[BV(\Omega)]^{m}$, where $\Omega$ is a bounded open set of $\mathbb{R}^{n}$ with Lipschitz boundary. Then there exists a subsequence $(u_{k_{\ell}})_{\ell\in\mathbb{N}}$ and a function $u\in [BV(\Omega)]^{m}$ such that $(u_{k_{\ell}})_{\ell\in\mathbb{N}}$ converges to $u$ weakly$^{\ast}$ in $[BV(\Omega)]^{m}$.

We say that the sequence $(u_{k})_{k\in \mathbb{N}}$ converges to $u$ \emph{strictly} in $[BV(\Omega)]^{m}$ if it converges to $u$ in $[L^{1}(\Omega)]^{m}$ and  $(|Du_{k}|(\Omega))_{k\in\mathbb{N}}$ converges to $|Du|(\Omega)$. It is immediate that strict convergence implies weak$^{\ast}$ convergence. 

Suppose that $\Omega$ is bounded  with Lipschitz boundary.  Define $1^{\ast}=n/(n-1)$
when $n>1$ and $1^{\ast}=\infty$ when $n=1$. Then $BV(\Omega)\subseteq L^{1^{\ast}}(\Omega)$ with continuous embedding. Moreover if $\Omega$ is connected then the following inequality holds (Poincar\'e-Wirtinger):
\[\|u-u_{\Omega}\|_{L^{1^{\ast}}(\Omega)}\le C|Du|(\Omega), \quad \forall u\in BV(\Omega),\]
where the constant $C$ depends only on $\Omega$ and  $u_{\Omega}$ denotes the mean value of $u$ in $\Omega$:
 \[u_{\Omega}:=\frac{1}{|\Omega|}\int_{\Omega}u~dx.\]
 We refer the reader to \cite{AmbrosioBV} for a detailed description of the above as well as for an introduction to weak continuity and differentiability notions in $BV(\Omega)$  and the decomposition of the distributional derivative of a function $u\in BV(\Omega)$.

\subsection{Relaxed functionals}\label{relaxedsection} Suppose that $X$ is a set endowed with some topology $\tau$ and let $F:X\to \mathbb{R}\cup\{+\infty\}$. The \emph{relaxed functional} or otherwise called the \emph{lower semicontinuous envelope} 
of $F$ with respect to the topology $\tau$ is a functional $\overline{F}:X\to\mathbb{R}\cup\{+\infty\}$ defined as follows for every $x\in X$:
\begin{eqnarray*}
\overline{F}(x)=\sup\{G(x):\; G:X\to \mathbb{R}\cup\{+\infty\},\;\tau\text{ lower semicontinuous, } G(y)\le F(y),\;\forall y\in X\}.
\end{eqnarray*}
It is easy to check that $\overline{F}$ is the greatest $\tau$ lower semicontinuous functional which is smaller or equal than $F$. It can be also checked that
\[\overline{F}(x)=\sup_{U\in \mathcal{N}(x)}\inf_{y\in U}F(y),\]
where $\mathcal{N}(x)$ denotes the open neighbourhoods of $x$. Moreover, if $X$ is a first countable topological space, then $\overline{F}(x)$ is characterised by the two following properties:
\begin{enumerate}
\item For every sequence $(x_{k})_{k\in \mathbb{N}}$ converging to $x$, we have
\[\overline{F}(x)\le \liminf_{k\to\infty}F(x_{k}).\]
\item There exists a sequence $(x_{k})_{k\in \mathbb{N}}$ converging to $x$ such that
\[\overline{F}(x)\ge \limsup_{k\to\infty}F(x_{k}).\]
\end{enumerate}
An interesting property of the relaxed functional is that if it has a minimum point then the value of $\overline{F}$ at that point will be equal with the infimum of $F$, i.e.,
\[\min_{x\in X}\overline{F}(x)=\inf_{x\in X}F(x).\]
For more information on relaxed functionals see  \cite{Braidesgamma} and \cite{dalmasogamma}.

\section{The variational formulation}\label{relax2order}
 In the current section we specify our definition of the functional \eqref{functional} that we want to minimise.  We start by defining the minimisation problem in the space $W^{2,1}(\Omega)$ as this is the space in which our analysis subsumes various choices for the regularisers $f$ and $g$. As this space is not reflexive, and thus existence of minimisers cannot be guaranteed, we extend the definition to a larger space. We introduce this larger space $BH(\Omega)$ as the subspace of all $u\in W^{1,1}(\Omega)$ such that $\nabla u\in [BV(\Omega)]^{m}$. We define the weak$^\ast$ and the strict topology of $BH(\Omega)$ and we identify the lower semicontinuous envelope (relaxed functional) of the extended functional with respect to these topologies. We prove existence of minimisers of the relaxed functional, uniqueness under some assumptions as well as stability.   

In the following $\Omega$  denotes as usual a bounded, connected, open subset of $\mathbb{R}^{2}$ with Lipschitz boundary, $T$  denotes a bounded linear operator from $L^{2}(\Omega)$ to $L^{2}(\Omega)$, $u_{0}\in L^{2}(\Omega)$ and $\alpha$, $\beta$ are non-negative constants. Further we suppose that $f:\mathbb{R}^{2}\to\mathbb{R}^{+}$, $g:\mathbb{R}^{4}\to\mathbb{R}^{+}$ are  convex functions with at most linear growth at infinity, i.e., there exist  positive constants $K_{1}$ and $K_{2}$ such that
\begin{eqnarray}
&&f(x)\le K_{1}(1+|x|),\quad \forall x\in\mathbb{R}^{2},\label{atmostlinear1}\\&& g(x)\le K_{2}(1+|x|),\quad \forall x\in\mathbb{R}^{4}.\label{atmostlinear2}
\end{eqnarray}
Moreover, we assume that both $f$ and $g$ are minimised in $0$ and they satisfy a coercivity condition:
\begin{eqnarray}
&&f(x)\ge K_{3}|x|,\quad \forall x\in\mathbb{R}^{2},\label{coercivity1}\\&& g(x)\ge K_{4}|x|,\quad \forall x\in\mathbb{R}^{4},\label{coercivity2}
\end{eqnarray}
where $K_{3}$ and $K_{4}$ are strictly positive.
We want to minimise the following functional:
\begin{eqnarray}\label{functional1}
H(u)=\frac{1}{2}\int_{\Omega}(u_{0}-Tu)^{2}~dx+\alpha\int_{\Omega}f(\nabla u)~dx+\beta\int_{\Omega}g(\nabla^{2}u)~dx.
\end{eqnarray}
The natural space for the functional $H$ to be defined in, is $W^{2,1}(\Omega)$. Since this space is not reflexive a solution of the minimisation problem by the direct method of calculus of variations does not work. Rather, existence of a minimiser of (\ref{functional1}) can be shown via relaxation that is: We extend the functional $H$ into a larger space which has some useful compactness properties with respect to some topology and we identify the relaxed functional with respect to the same topology. This space is $BH(\Omega)$. 

\subsection{The space $BH(\Omega)$} The space $BH(\Omega)$ (often denoted with $BV^{2}(\Omega)$) is the space of functions of bounded Hessian. It was introduced by Demengel in \cite{demengelBH} and consists of all functions $u\in W^{1,1}(\Omega)$ whose distributional Hessian can be represented by an $\mathbb{R}^{2}\times \mathbb{R}^{2}$-valued finite Radon measure. In other words:
\[BH(\Omega)=\{u\in W^{1,1}(\Omega):\; \nabla u\in [BV(\Omega)]^{2}\}.\]
We set $D^{2}u:=D(\nabla u)$. Again it is immediate that $W^{2,1}(\Omega)\subseteq BH(\Omega)$. $BH(\Omega)$ is a Banach space equipped with the norm $\|u\|_{BH(\Omega)}=\|u\|_{BV(\Omega)}+|D^{2}u|(\Omega)$. If $\Omega$ has a Lipschitz boundary and it is connected then it can be shown that there exist positive constants $C_{1}$ and $C_{2}$ such that
\begin{equation}\label{BHinterp}
\int_{\Omega}|\nabla u|~dx\le C_{1}|D^{2}u|(\Omega)+C_{2}\int_{\Omega}|u|~dx,\;\; \forall u\in BH(\Omega).
\end{equation}
Moreover, the embedding from $BH(\Omega)$ into $W^{1,1}(\Omega)$ is compact, see \cite{demengelBH}.
We denote the approximate differential of $\nabla u$ with $\nabla^{2}u$, see \cite{AmbrosioBV} for a definition.

Analogously with $BV(\Omega)$ we define the following notions of convergence in $BH(\Omega)$:
\newtheorem{first}{Definition}[section]
\begin{first}[\textbf{Weak$^\ast$ convergence in $\mathbf{BH(\Omega)}$}]\label{weakBH}
Let $(u_{k})_{k\in\mathbb{N}},\;u$  belong to $BH(\Omega)$. We say that $(u_{k})_{k\in\mathbb{N}}$ converges to $ u$ weakly$^\ast$ in $BH(\Omega)$ if 
\[u_{k}\to u,\quad \text{in }L^{1}(\Omega)\] and \[ \nabla u_{k}\rightharpoonup \nabla u\quad \text{weakly}^{\ast}\text{ in }[BV(\Omega)]^{2},\quad \text{as }k\to\infty,\]
or in other words
\[\|u_{k}-u\|_{L^{1}(\Omega)}\to 0,\]\[ \|\nabla u_{k}-\nabla u\|_{[L^{1}(\Omega)]^{2}}\to0,\]\[ \int_{\Omega}\phi~ dD^{2}u_{k}\to \int_{\Omega}\phi~ dD^{2}u,\;\;\forall \phi\in C_{0}(\Omega).\]
\end{first}

It is not hard to check that a basis for that topology consists of the following sets:
\begin{eqnarray*}
U(v,F,\epsilon)&=&\Bigg \{u\in BH(\Omega):\; \|v-u\|_{L^{1}(\Omega)}+\|\nabla v-\nabla u\|_{[L^{1}(\Omega)]^{2}}+\left |\int_{\Omega}\phi_{i}~dD^{2}v-\int_{\Omega}\phi_{i}~dD^{2}u \right |<\epsilon,\;i\in F\Bigg \},
\end{eqnarray*}
where $v\in BH(\Omega) $, $F$ is finite, $\epsilon>0$ and $\phi_{i}\in C_{0}(\Omega)$. This topology is also metrizable, hence first countable. We do not imply here that $BH(\Omega)$ is the dual space of a Banach space but we name this convergence $weak^{\ast}$ to show the correspondence with the weak$^\ast$ convergence in $BV(\Omega)$. We have also the corresponding compactness result:
\newtheorem{BHcompact}[first]{Theorem}
\begin{BHcompact}[\textbf{Compactness in }$\mathbf{BH(\Omega)}$]\label{BHcompact}
Suppose that the sequence $(u_{k})_{k\in\mathbb{N}}$ is bounded in $BH(\Omega)$. Then there exists a subsequence $(u_{k_{\ell}})_{\ell\in \mathbb{N}}$ and a function $u\in BH(\Omega)$ such that $(u_{k_{\ell}})_{\ell\in\mathbb{N}}$ converges to $u$ weakly$^{\ast}$ in $BH(\Omega)$.
\end{BHcompact}
\begin{proof}
From the compact embedding of $BH(\Omega)$ into $W^{1,1}(\Omega)$ and the fact that the sequence $(\nabla u_{k})_{k\in\mathbb{N}}$ is bounded in $[BV(\Omega)]^{2}$ we have that there exists a subsequence $(u_{k_{\ell}})_{\ell\in\mathbb{N}}$, a function $u\in W^{1,1}(\Omega)$ and a function $v\in [BV(\Omega)]^{2}$ such that $(u_{k_{\ell}})_{\ell\in\mathbb{N}}$ converges to $u$ in $W^{1,1}(\Omega)$ and $(\nabla u_{k_{\ell}})_{\ell\in\mathbb{N}}$ converges to $v$ weakly$^{\ast}$ in $[BV(\Omega)]^{2}$, as $\ell$ goes to infinity. Then, $\nabla u=v$, $u\in BH(\Omega)$ and $(u_{k_{\ell}})_{\ell\in\mathbb{N}}$ converges to $u$ weakly$^\ast$ in $BH(\Omega)$.
 \end{proof}

\newtheorem{strictBH}[first]{Definition}
\begin{strictBH}[\textbf{Strict convergence in $\mathbf{BH}$}]
Let $(u_{k})_{k\in\mathbb{N}}$ and $u$  belong to $BH(\Omega)$. We  say that $(u_{k})_{k\in\mathbb{N}}$ converges to $u$ strictly in $BH(\Omega)$ if
\[u_{k}\to u,\quad \text{in }L^{1}(\Omega)\] and \[|D^{2}u_{k}|(\Omega)\to |D^{2}u|(\Omega),\quad\text{as }k\to\infty.\]
\end{strictBH}
\noindent It is easily checked that the function 
\[d(u,v)=\int_{\Omega}|u-v|~dx+\left ||D^{2}u|(\Omega)-|D^{2}v|(\Omega) \right |,\]
 is a metric and induces the strict convergence in $BH(\Omega)$. The following Lemma can be used to compare these two topologies.

\newtheorem{interpolation}[first]{Lemma}
\begin{interpolation}\label{interpolation}
 Suppose that $(u_{k})_{k\in \mathbb{N}}$, $u^{\ast}$ belong to $BH(\Omega)$ and $(u_{k})_{k\in\mathbb{N}}$ converges to $u^{\ast}$ strictly in $BH(\Omega)$. Then
\[\|u_{k}-u^{\ast}\|_{W^{1,1}(\Omega)}\to 0, \quad\text{as }k\to\infty.\]
\end{interpolation}

\begin{proof}
We recall from (\ref{BHinterp}) that there exist  positive constants $C_{1}$ and $C_{2}$ such that
\[\int_{\Omega}|\nabla u|~dx\le C_{1}|D^{2}u|(\Omega)+C_{2}\int_{\Omega}|u|~dx,\;\; \forall u\in BH(\Omega).\]
Since the sequence $(u_{k})_{k\in\mathbb{N}}$ is strictly convergent in $BH(\Omega)$, the sequences $(\|u_{k}\|_{L^{1}(\Omega)})_{k\in\mathbb{N}}$ and $(|D^{2}u_{k}|(\Omega))_{k\in\mathbb{N}}$ are bounded. Hence, there exists a positive constant $C$ such that
\[\int_{\Omega}|\nabla u_{k}|~dx<C,\quad \forall k \in \mathbb{N},\]
which implies that the sequence $(u_{k})_{k\in\mathbb{N}}$ is bounded in $BH(\Omega)$. From the compact embedding of $BH(\Omega)$ into $W^{1,1}(\Omega)$ we get that there exists a subsequence $(u_{k_{\ell}})_{\ell\in\mathbb{N}}$ and a function $v\in W^{1,1}(\Omega)$ such that $(u_{k_{\ell}})_{\ell\in\mathbb{N}}$ converges to $v$ in $W^{1,1}(\Omega)$. In particular $(u_{k_{\ell}})_{\ell\in\mathbb{N}}$ converges to $v$ in $L^{1}(\Omega)$ so $v=u^{\ast}$ and thus $(u_{k_{\ell}})_{\ell\in\mathbb{N}}$ converges to $u^{\ast}$ in $W^{1,1}(\Omega)$. Since every subsequence of $(u_{k})_{k\in\mathbb{N}}$ is bounded in $BH(\Omega)$ we can repeat the same argument and deduce that for every subsequence of $(u_{k})_{k\in\mathbb{N}}$ there exists a further subsequence which converges to $u^{\ast}$ in $W^{1,1}(\Omega)$. This proves that the initial sequence $(u_{k})_{k\in\mathbb{N}}$ converges to $u^{\ast}$ in $W^{1,1}(\Omega)$.
\end{proof}

\newtheorem{cor}[first]{Corollary}
\begin{cor}
Strict convergence implies weak$^{\ast}$ convergence in $BH(\Omega)$.
\end{cor}

\subsection{Relaxation of the second order functional}

We now extend the functional $H$ in \eqref{functional1} to $BH(\Omega)$ in the following way:
\begin{equation}\label{functional2}
H_{ex}(u)=
\begin{cases}
\frac{1}{2}\int_{\Omega}(u_{0}-Tu)^{2}~dx+\alpha\int_{\Omega}f(\nabla u)~dx+\beta\int_{\Omega}g(\nabla^{2}u)~dx & \text{ if }u\in W^{2,1}(\Omega),\\
+\infty & \text{ if }f\in \;\; BH(\Omega)\setminus W^{2,1}(\Omega).
\end{cases}
\end{equation}
As we have discussed above, the weak$^{\ast}$ topology in $BH(\Omega)$  provides a good compactness property which is inherited from the weak$^{\ast}$ topology in $[BV(\Omega)]^{n}$.  However, the functional $H_{ex}$ is not sequentially lower semicontinuous with respect to the strict topology in $BH(\Omega)$ and hence it is neither with respect to the weak$^{\ast}$ topology in $BH(\Omega)$. Indeed, we can find a  function $u\in BH(\Omega)\setminus W^{2,1}(\Omega)$, see \cite{demengelBH} for such an example. Hence, from the definition of $H_{ex}$ we have $H_{ex}(u)=\infty$. However, according to Theorem \ref{demtem} we can find a sequence $(u_{k})_{k\in\mathbb{N}}$ in $W^{2,1}(\Omega)$ which converges strictly to $u$. It follows that the sequences $(\|u_{k}\|_{L^{1}(\Omega)})_{k\in\mathbb{N}}$, $(|D^{2}u_{k}|(\Omega))_{k\in\mathbb{N}}$ as well as $(\|\nabla u_{k}\|_{L^{1}(\Omega)})_{k\in\mathbb{N}}$ are bounded. Moreover the sequence $(u_{k})_{k\in\mathbb{N}}$ is bounded in $L^{2}(\Omega)$. Since $T$ is a bounded linear operator and from the fact that $f$ and $g$ have at most linear growth at infinity we deduce that the sequence $(H_{ex}(u_{k}))_{k\in\mathbb{N}}$ is bounded as well. Hence we get
\[H_{ex}(u)>\liminf_{k\to\infty}H_{ex}(u_{k}),\]
which proves that
 $H_{ex}$ is not lower semicontinuous with respect to the strict topology in $BH(\Omega)$. Thus, we have to identify its lower semicontinuous envelope with respect to the strict convergence. We define the following functional in $BH(\Omega)$ :
\begin{eqnarray*}\label{functional3}
\overline{H_{ex}}(u)&:=&\frac{1}{2}\int_{\Omega}(u_{0}-Tu)^{2}dx+\alpha\int_{\Omega}f(\nabla u)~dx+\beta \,g(D^{2}u)(\Omega)\\
&=&\frac{1}{2}\int_{\Omega}(u_{0}-Tu)^{2}dx+\alpha\int_{\Omega}f(\nabla u)~dx+\beta\int_{\Omega}g(\nabla^{2} u)~dx+\beta\int_{\Omega}g_{\infty}\left (\frac{D^{s}\nabla u}{|D^{s}\nabla u|} \right )d|D^{s}\nabla u|,
\end{eqnarray*}
where $\nabla^{2}u$, the approximate differential of $\nabla u$, is also the density of $D^{2}u$ with respect to the Lebesgue measure, see \cite{AmbrosioBV}.
It is immediate to see that if $u\in W^{2,1}(\Omega)$ then $\overline{H}_{ex}(u)=H_{ex}(u)$. Thus is general, $\overline{H}_{ex}$ is smaller than $H_{ex}$.

\newtheorem{lscsecond}[first]{Theorem}
\begin{lscsecond}\label{lsc2}
The functional  $\overline{H_{ex}}$ is lower semicontinuous with respect to the strict topology in $BH(\Omega)$. 
\end{lscsecond}

\begin{proof}
It is not hard to check that since $f$ is convex and has at most linear growth then it is Lipschitz, say with constant $L>0$. Let $u$ and $(u_{k})_{k\in\mathbb{N}}$ be functions in $BH(\Omega)$ and let $(u_{k})_{k\in\mathbb{N}}$ converge to $u$ strictly in $BH(\Omega)$ and thus also weakly$^{\ast}$ in $BH(\Omega)$.  We have to show that
\[\overline{H_{ex}}(u)\le \liminf_{k\to\infty}\overline{H_{ex}}(u_{n}).\]
 From the definition of the weak$^{\ast}$ convergence in $BH(\Omega)$ we have that $(u_{k})_{k\in\mathbb{N}}$ converges to $u$ in $W^{1,1}(\Omega)$. From the Sobolev inequality, see \cite{EvansPDEs},
 \[\|v\|_{L^{2}(\Omega)}\le C\|v\|_{W^{1,1}(\Omega)},\quad \forall v\in W^{1,1}(\Omega),\]
we have that $(u_{k})_{k\in\mathbb{N}}$ converges to $u$ in $L^{2}(\Omega)$. Since $T:L^{2}(\Omega)\to L^{2}(\Omega)$ is continuous then the map $u\mapsto \frac{1}{2}\int_{\Omega}(u_{0}-Tu)^{2}dx$ is continuous and hence we have that
\begin{equation}\label{1over3}
\frac{1}{2}\int_{\Omega}(u_{0}-Tu_{k})^{2}dx\to \frac{1}{2}\int_{\Omega}(u_{0}-Tu)^{2}dx,\quad \text{as }k\to\infty.
\end{equation}
Moreover since $\|\nabla u_{k}-\nabla u\|_{[L^{1}(\Omega)]^{2}}$ converges to $0$ as $k\to\infty$, we have from the Lipschitz property
\[\left |\int_{\Omega}f(\nabla u_{k})~dx-\int_{\Omega}f(\nabla u)~dx \right |\le\]\[\int_{\Omega}\left |f(\nabla u_{k})-f(\nabla u)\right |~dx 
\le\]\[ L\int_{\Omega}|\nabla u_{k}-\nabla u|~dx\to 0,\quad \text{as }k\to\infty,\]
i.e., we have
\begin{equation}\label{2over3}
\int_{\Omega}f(\nabla u_{k})~dx\to \int_{\Omega}f(\nabla u)~dx,\quad \text{as }k\to\infty.
\end{equation}
Finally we have that $D^{2}u_{k}\to D^{2} u $ weakly$^{\ast}$. We can apply Theorem \ref{BF} for $\mu_{k}=\mu=\mathcal{L}^{2}$, $\nu=D^{2}u$, $\nu_{k}=D^{2}u_{k}$ and get that
\begin{equation}\label{3over3}
g(D^{2}u)(\Omega)\le \liminf_{k\to\infty}g(D^{2}u_{k})(\Omega).
\end{equation}
From (\ref{1over3}), (\ref{2over3}) and (\ref{3over3}) we get that
\[\overline{H_{ex}}(u)\le \liminf_{k\to\infty}\overline{H_{ex}}(u_{k}).\]
\end{proof}

\newtheorem{relaxed}[first]{Theorem}
\begin{relaxed}\label{relaxationresult}
The functional $\overline{H_{ex}}$ is the lower semicontinuous envelope of $H_{ex}$ with respect to the strict convergence in $BH(\Omega)$.
\end{relaxed}

\begin{proof}
Suppose that $(u_{k})_{k\in\mathbb{N}}$ converges to $u$ strictly in $BH(\Omega)$. From the lower semicontinuity of $\overline{H_{ex}}$ we have that $\overline{H_{ex}}(u)\le \liminf_{k\to\infty}\overline{H_{ex}}(u_{k})$ and since $\overline{H_{ex}}\le H_{ex}$ we get that $\overline{H_{ex}}(u)\le \liminf_{k\to\infty}H_{ex}(u_{k})$. For the other direction Theorem \ref{demtem} tells us that given $u\in BH(\Omega)$ there exist a sequence $(u_{k})_{k\in\mathbb{N}}\subseteq C^{\infty}(\Omega)\cap W^{2,1}(\Omega)$ such that $(u_{k})_{k\in\mathbb{N}}$ converges to $u$ strictly in $BH(\Omega)$ and
\begin{equation}\label{eq3}
g(D^{2}u_{k})(\Omega)\to g(D^{2}u)(\Omega).
\end{equation}
We have also that $(u_{k})_{k\in\mathbb{N}}$ converges to $u$ in $W^{1,1}(\Omega)$ which implies that
\begin{eqnarray}\label{eq4}
&&\frac{1}{2}\int_{\Omega}(u_{0}-Tu_{k})^{2}dx+\alpha \int_{\Omega}f(\nabla u_{k})~dx\to\nonumber \\&& \frac{1}{2}\int_{\Omega}(u_{0}-Tu)^{2}dx+\alpha \int_{\Omega}f(\nabla u)~dx,\\ &&\ \text{as }k\to\infty,\nonumber
\end{eqnarray}
as the proof of Theorem \ref{lsc2} shows.
From (\ref{eq3}) and (\ref{eq4}) we get that
\[\overline{H_{ex}}(u)= \lim_{k\to\infty}H_{ex}(u_{k}).\]
\end{proof}

 Observe that $\overline{H_{ex}}$ is also the lower semicontinuous envelope of $H_{ex}$ with respect to the weak$^{\ast}$ convergence in $BH(\Omega)$. 
 
Let us note here that in fact the relaxation result of Theorem \ref{relaxationresult} follows from a more general relaxation result in \cite{amar1994relaxation}. There, the authors solely assume $g$ to be  quasi-convex. However, since we consider convex functions the proof we give in this paper is simpler and more accessible to the non-specialist reader.

The proof of the following minimisation theorem follows the proof of the corresponding theorem in \cite{vese2001study} for the analogue first order functional. Here we denote with $\mathcal{X}_{\Omega}$ the characteristic function of $\Omega$, i.e., $\mathcal{X}_{\Omega}(x)=1$, for all $x\in\Omega$ and 0 otherwise. 
\newtheorem{existence}[first]{Theorem}
\begin{existence}\label{Theoremexistence}
Assuming $T(\mathcal {X}_{\Omega})\ne 0$, $\alpha>0$, $\beta>0$ then the minimisation problem
\begin{equation}\label{existence}
\inf_{u\in BH(\Omega)}\overline{H_{ex}}(u),
\end{equation}
has a solution $u\in BH(\Omega)$.
\end{existence}

\begin{proof}
Let $(u_{k})_{k\in\mathbb{N}}$ be a minimising sequence for (\ref{existence}) and let $C>0$ be an upper bound for $(\overline{H_{ex}}(u_{k}))_{k\in\mathbb{N}}$. We have that
\begin{equation}\label{bound1}
\int_{\Omega}f(\nabla u_{k})~dx<C\quad \text{and}\quad \frac{1}{2}\int_{\Omega}(u_{0}-Tu_{k})^{2}dx<C,
\end{equation}
for every $k\in\mathbb{N}$. 
From the coercivity assumptions \eqref{coercivity1}-\eqref{coercivity2} and   from (\ref{bound1}) we have
\begin{equation}\label{bound2}
|Du_{k}|(\Omega)=\int_{\Omega}|\nabla u_{k}|~dx  < C,\quad \forall k\in \mathbb{N},
\end{equation}
for a possibly different constant $C$.
 We show that the sequence $(u_{k})_{k\in\mathbb{N}}$ is bounded in $L^{2}(\Omega)$, following essentially \cite{vese2001study}.
 By the  Poincar\'e-Wirtinger inequality there exists a positive constant $C_{1}$ such that for every $k\in\mathbb{N}$
 \begin{eqnarray*}
\|u_{k}\|_{L^{2}(\Omega)}&=&\Bigg\|u_{k}-\mathcal{X}_{\Omega}\frac{1}{|\Omega|}\int u_{k}~dx +\mathcal{X}_{\Omega}\frac{1}{|\Omega|}\int u_{k}~dx\Bigg\|_{L^{2}(\Omega)} \\&\le& C_{1} |Du_{k}|(\Omega)+\left |\int_{\Omega}u_{k}~dx \right |\\&\le& C+\left |\int_{\Omega}u_{k}~dx \right |.
\end{eqnarray*}
Thus it suffices to bound $\left|\int_{\Omega}u_{k}~dx\right|$ uniformly in $k$. We have for every $k\in \mathbb{N}$
\begin{eqnarray*}
\left \|T\left(\mathcal{X}_{\Omega}\frac{1}{|\Omega|}\int u_{k}~dx \right )  \right \|_{L^{2}(\Omega)}&\le&
\left \|T\left(\mathcal{X}_{\Omega}\frac{1}{|\Omega|}\int u_{k}~dx \right ) -Tu_{k}\right\|_{L^{2}(\Omega)}+\left\|Tu_{k}-u_{0}\right\|_{L^{2}(\Omega)}+\|u_{0}  \|_{L^{2}(\Omega)}\\
&\le&\|T\|\left\| u_{k}-\mathcal{X}_{\Omega}\frac{1}{|\Omega|}\int u_{k}~dx  \right\|_{L^{2}(\Omega)}+\left\|Tu_{k}-u_{0}\right\|_{L^{2}(\Omega)}+\|u_{0}  \|_{L^{2}(\Omega)}\\
&\le&C_{1}\|T\||Du_{k}|(\Omega)+\sqrt{2C}+\|u_{0}  \|_{L^{2}(\Omega)}\\
&\le&C_{1}\|T\|C+\sqrt{2C}+\|u_{0}  \|_{L^{2}(\Omega)}:=C'.
\end{eqnarray*}
It follows that 
 \[\left |\int_{\Omega}u_{k}~dx\right|\|T(\mathcal{X}_{\Omega})\|_{L^{2}(\Omega)}\le C'|\Omega|,\]
 and thus
\[\left |\int_{\Omega}u_{k}~dx\right|\le\frac{C'|\Omega|}{\|T(\mathcal{X}_{\Omega})\|_{L^{2}(\Omega)}},\]
since $T(\mathcal{X}_\Omega)\ne0$.

 Since the sequence is bounded in $L^{2}(\Omega)$ and $\Omega$ is bounded, we have that the sequence is bounded in $L^{1}(\Omega)$ and moreover it is bounded in $BH(\Omega)$. From  Theorem \ref{BHcompact} we obtain the existence of a subsequence $(u_{k_{\ell}})_{\ell\in\mathbb{N}}$ and $u\in BH(\Omega)$ such that $(u_{k_{\ell}})_{\ell\in\mathbb{N}}$ converges to $u$ weakly$^{\ast}$ in $BH(\Omega)$. Since the functional $\overline{H_{ex}}$ is lower semicontinuous with respect to this convergence we have:
 \[\overline{H_{ex}}(u)\le \liminf_{k\to\infty}\overline{H_{ex}}(u_{k})\]
 which implies that
 \[u=\min_{u\in BH(\Omega)}\overline{H_{ex}}(u).\]

\end{proof}

Let us note here that in the above proof we needed $\alpha>0$, in order to get an a priori bound in the $L^{1}$ norm of the gradient (for $\beta=0$ see \cite{vese2001study}). However, the proof goes through if $\alpha=0$ and $T$ is injective. If $T$ is not injective and $\alpha=0$ it is not straightforward how to get existence. 
The proof of the following theorem also follows the proof of the corresponding theorem for the first order analogue in \cite{vese2001study}.
\newtheorem{uniqueness}[first]{Proposition}
\begin{uniqueness}\label{uniqueness}
If, in addition to $T(\mathcal{X}_{\Omega})\ne 0$,  $T$ is injective or if $f$ is strictly convex, then the solution of the minimisation problem (\ref{existence}) is unique.
\end{uniqueness}

\begin{proof}
Using the Proposition \ref{convexmeasure} in Appendix \ref{appendixA} we can check that the functional $\overline{H_{ex}}$ is convex. Let $u_{1}$, $u_{2}$ be two minimisers. If $T(u_{1})\ne T(u_{2})$ then from the strict convexity of the first term of $\overline{H_{ex}}$ we have 
\[\overline{H_{ex}}\left (\frac{1}{2}u_{1} +\frac{1}{2}u_{2} \right )<\frac{1}{2}\overline{H_{ex}}(u_{1})+\frac{1}{2}\overline{H_{ex}}(u_{2})=\inf_{u\in BH(\Omega)}\overline{H_{ex}}(u),\]
which is a contradiction. Hence $T(u_{1})= T(u_{2})$ and if $T$ is injective we have $u_{1}=u_{2}$. If $T$ is not injective but $f$ is strictly convex then we must have $\nabla u_{1}=\nabla u_{2}$ otherwise we get the same contradiction as before. In that case, since $\Omega$ is connected, there exists a constant $c$ such that $u_{1}=u_{2}+c\mathcal{X}_{\Omega}$ and since $T(\mathcal{X}_{\Omega})\ne 0$, we get $c=0$.
\end{proof}

\subsection{Stability}
To complete the well-posedness picture for \eqref{functional} it remains to analyse the stability of the method. More precisely, we want to know which effect deviations in the data $u_0$ have on a corresponding minimiser of \eqref{functional}. Ideally the deviation in the minimisers for different input data should be bounded by the deviation in the data. Let $R$ be the regularising functional in \eqref{functional}, i.e.,
$$
R(u) = \alpha\int_\Omega f(\nabla u)~ dx + \beta\int_\Omega g(\nabla^2 u)~ dx.
$$
It has been demonstrated by many authors \cite{burger2004convergence,burger2007error,poschl2008tikhonov,benning2011error} that Bregman distances related to the regularisation functional $R$ are natural error measures for variational regularisation methods with $R$ convex. In particular P\"oschl \cite{poschl2008tikhonov} has derived estimates for variational regularisation methods with powers of metrics, which apply to the functional we consider here. However, for demonstration issues and to make constants in the estimates more explicit let us state and prove the result for a special case of \eqref{functional} here. 

We consider functional \eqref{functional} for the case $s=2$. For what we are going to do we assume that one of the regularisers is differentiable. Without loss of generality, we assume that $f(s)$ is differentiable in $s$. The analogous analysis can be done if $g(s)$ is differentiable or even under weaker, continuity conditions on $f$ and $g$, see \cite{bredies2011}. Let $\tilde u$ be the original image and $\tilde u_0$ the exact datum (without noise), i.e. $\tilde u_0$ is a solution of $T\tilde u=\tilde u_0$. We assume that the noisy datum $u_0$ deviates from the exact datum by $\|\tilde u_0-u_0\|_{L^2}\leq \delta$ for a small $\delta>0$. For the original image $\tilde u$ we assume that the following condition, called source condition, holds
\begin{equation}
\tag{SC}
\textrm{There exists a } \xi\in\partial R(\tilde u) \textrm{ such that } \xi = T^* q
 \textrm{ for a source element } q\in D(T^*),
 \label{eq:SC}
\end{equation}
where $D(T^*)$ denotes the domain of the operator $T^*$ and $\partial R$ is the subdifferential of $R$. Moreover, since $f$ is differentiable and both $f$ and $g$ are convex, the subdifferential of $R$ can be written as the sum of subdifferentials of the two regularisers, i.e.,
\begin{align*}
\partial R(u) &= \alpha \partial \left(\int_\Omega f(\nabla u)~ dx\right) + \beta \partial\left(\int_\Omega g(D^2 u)\right)\\
& = -\alpha \mathrm{div}(f'(\nabla u)) + \beta \partial\left(\int_\Omega g(D^2 u)\right),
\end{align*}
see \cite{ekeland1999convex}[Proposition 5.6., pp. 26]. We define the symmetric Bregman distance for the regularising functional $R$ as
\begin{equation*}
D_R^{symm}(u_1,u_2) := \langle p_1 - p_2, u_1-u_2 \rangle, 
\quad p_1\in\partial R(u_1), \;  p_2\in\partial R(u_2).
\end{equation*}
\newtheorem{stability}[first]{Theorem}
\begin{stability}
Let $\tilde u$ be the original image with source condition \eqref{eq:SC} satisfying $T\tilde u = \tilde u_0$. Let $u_0\in L^2(\Omega)$ be the noisy datum with $\|\tilde u_0-u_0\|_{L^2}\leq \delta$. Then a minimiser $u$ of \eqref{functional} fulfils
\begin{equation*}
\alpha D^{symm}_{\int_\Omega f(\nabla\cdot)}(u,\tilde u) + \beta D^{symm}_{\int_\Omega g(\nabla^2\cdot)}(u,\tilde u) + \frac{1}{2}\|Tu-\tilde u_0\|_{L^2}^2
 \leq  2\alpha^{2} \|q_\nabla\|_{L^2}^2 + 2\beta^{2} \|q_{\nabla^2}\|_{L^2}^2 + \delta^2,
\end{equation*}
where the source element $q$ is decomposed in $q=\alpha q_\nabla + \beta q_{\nabla^2}$. Moreover, let $u_1$ and $u_2$ be minimisers of \eqref{functional} with data $u_{0,1}$ and $u_{0,2}$ respectively. Then the following estimate is true
\begin{eqnarray}
&&\alpha D^{symm}_{\int_\Omega f(\nabla\cdot)}(u_1,u_2) + \beta D^{symm}_{\int_\Omega g(\nabla^2\cdot)}(u_1,u_2) + \frac{1}{2}\|T(u_1-u_2)\|_{L^2}^2 \leq \frac{1}{2}\|u_{0,1}- u_{0,2}\|_{L^2}^2.\label{istrue}
\end{eqnarray}
\end{stability}
\begin{proof}
The optimality condition for \eqref{functional} for $s=2$ and differentiable $f$ reads
\[
\alpha p_\nabla + \beta p_{\nabla^2} + T^*(Tu-u_0) = 0,
 \]
\[
\quad p_\nabla = -\mathrm{div}(f'(\nabla u)), \; p_{\nabla^2}\in \partial\left(\int_\Omega g(\nabla^2 u)\right).
\]
Adding $\xi$ from \eqref{eq:SC} and $T^*\tilde u_0$ we get
$$
\alpha p_\nabla + \beta p_{\nabla^2} - \xi + T^*(Tu-\tilde u_0) = T^*((u_0-\tilde u_0)-q).
$$
Then we use that $\xi=\alpha\xi_\nabla + \beta \xi_{\nabla^2}$ and take the duality product with $u-\tilde u$, which gives
\begin{equation*}
\alpha \langle p_\nabla-\xi_\nabla,u-\tilde u \rangle + \beta \langle p_{\nabla^2}-\xi_{\nabla^2},u-\tilde u \rangle + \|Tu-\tilde u_0\|_{L^2}^2
 = \langle (u_0-\tilde u_0)-q,Tu-\tilde u_0\rangle.
\end{equation*}
By Young's inequality and the standard inequality $\frac{1}{2} (a+b)^2 \leq a^2 + b^2$ for $a,b\in\mathbb R$, we eventually get
\begin{equation*}
\alpha D^{symm}_{\int_\Omega f(\nabla\cdot)}(u,\tilde u) + \beta D^{symm}_{\int_\Omega g(\nabla^2\cdot)}(u,\tilde u) + \frac{1}{2}\|Tu-\tilde u_0\|_{L^2}^2
 \leq  2\alpha^2 \|q_\nabla\|_{L^2}^2 + 2\beta^2 \|q_{\nabla^2}\|_{L^2}^2 + \|u_0-\tilde u_0\|_{L^2}^2.
\end{equation*}

\noindent Similarly, we can derive the stability estimate by subtracting the optimality conditions for \eqref{functional} in $u_1$ and $u_2$ with datum $u_{0,1}$ and $u_{0,2}$, respectively and applying Young's inequality again, we get
\eqref{istrue}.
\end{proof}

\newtheorem{remarkstab}[first]{Remark}
\begin{remarkstab}
In the case $\alpha=\beta$ we obtain the classical form of Bregman error estimates, that is
\begin{equation*}
D^{symm}_{\int_\Omega f(\nabla\cdot)}(u,\tilde u) + D^{symm}_{\int_\Omega g(\nabla^2\cdot)}(u,\tilde u) + \frac{1}{2\alpha}\|Tu-\tilde u_0\|_{L^2}^2
 \leq  2\alpha \left(\|q_\nabla\|_{L^2}^2 +  \|q_{\nabla^2}\|_{L^2}^2\right) + \frac{1}{\alpha}\|u_0-\tilde u_0\|_{L^2}^2.
\end{equation*}
\end{remarkstab}

\section{Special cases and extensions}\label{specialsession}
In this section we introduce two more versions of the functional $\overline{H_{ex}}$, the anisotropic version and the version where the $L^{1}$ norm appears in the fidelity term.
\subsection{The anisotropic version}\label{unisosection}
We introduce the anisotropic version of  the functional $H$ in \eqref{functional1}. Note that when $f(x)=|x|$ , $g(x)=|x|$ then the relaxed functional $\overline{H_{ex}}$ is given by
$$
\frac{1}{2}\int_{\Omega}(u_{0}-Tu)^{2}dx+\alpha\int_{\Omega}|\nabla u| ~dx+\beta |D^2 u|{\Omega}.
$$
Its anisotropic analogue is defined for $f(x)=|x|_{1}$ and $g(x)=|x|_{1}$. In that case, the relaxed functional is given by
\begin{eqnarray}\label{anisofunctional}
F(u)&=&\frac{1}{2}\int_{\Omega} (u_{0}-Tu)^{2}~ dx+\alpha\int_{\Omega}(|u_{x}|+|u_{y}|)~dx \nonumber\\ &&+\beta \big(|D_{1}u_{x}|(\Omega)+|D_{2}u_{x}|(\Omega)+|D_{1}u_{y}|(\Omega)+|D_{2}u_{y}|(\Omega)\big ),
\end{eqnarray}
where $D_{i}$, $i=1,2$ denotes the distributional derivative with respect to $x$ and $y$ respectively.
 Since the functional $F$ is obtained for the above choice of $f$ and $g$, the following theorem holds as a special case of Theorem \ref{Theoremexistence}:

\newtheorem{anisotheorem}[first]{Theorem}
\begin{anisotheorem}\label{anisotheorem}
Assuming $T(\mathcal{X}_{\Omega})\ne 0$, the minimisation problem
\begin{equation}\label{minianiso}
\inf_{u\in BH(\Omega)}F(u),
\end{equation}
has a solution. If $T$ is injective then the solution is unique.
\end{anisotheorem}

\subsection{$L^{1}$ fidelity term}\label{sectionL1}
We  consider here the case with the $L^{1}$ norm in the fidelity term, i.e.,
\begin{equation}\label{L1functional}
G(u)=\int_{\Omega}|u_{0}-Tu|~dx+\alpha\int_{\Omega}|\nabla u|~dx+\beta|D^{2}u|(\Omega),
\end{equation}
where  for simplicity we consider the case $f(x)=|x|$, $g(x)=|x|$. As it has been shown in \cite{nikolova04} and also studied in \cite{chanL1} and \cite{duvalL1}, the $L^1$ norm in the fidelity term leads to efficient restorations of images that have been corrupted by impulse  noise.
\newtheorem{L1min}[first]{Theorem}
\begin{L1min}
Assuming $T(\mathcal{X}_{\Omega})\ne 0$, the minimisation problem
\begin{equation}\label{miniL1}
\inf_{u\in BH(\Omega)}G(u),
\end{equation}
has a solution.
\end{L1min}
\begin{proof}
The proof is another application of the direct method of calculus of variation. Similarly with the proof of Theorem \ref{Theoremexistence} we show that  any minimising sequence is bounded in $L^{1}(\Omega)$. Hence it is bounded in $BH(\Omega)$. Thus we can extract a weakly$^{\ast}$ convergent subsequence in $BH(\Omega)$. Trivially, the functional is lower semicontinuous with respect to that convergence.
\end{proof}
Note that in this case the uniqueness of the minimiser cannot be guaranteed since the functional $G$ is not strictly convex anymore, even in the case where $T=Id$. The versions with general $f$ and $g$ of Theorem \ref{Theoremexistence} can be easily extended to the cases discussed in Sections 
\ref{unisosection} and \ref{sectionL1}.
\section{The numerical implementation}\label{numericssection}
In this section we work with the discretised version of the functional \eqref{functional1} and we discuss its numerical realisation by the so-called split Bregman technique \cite{goldstein2009split}. We start   by defining the discrete versions of $L^{1}$ and $L^{2}$ norms in  Section \ref{disset}. In Section \ref{seccon} we proceed with an introduction to the Bregman iteration which is used to solve constrained optimisation problems, an idea originated in \cite{OBG}. In \cite{goldstein2009split} the Bregman iteration  and an operator splitting technique (split Bregman) is used in order to solve the total variation minimisation problem. In the latter paper it was also proved that the iterates of the Bregman iteration converge to the solution of the constrained problem assuming that the iterates satisfy the constraint in a finite number of iterations. Here, we give a more general convergence result  where we do not use that assumption, see Theorem \ref{bregmanprop}. Finally in  Section \ref{Bregmanourcase} we describe how our problem is solved with the Bregman iteration, using the splitting procedure mentioned above.
\subsection{The discretised setting}\label{disset}
In this section we study the discretisation and minimisation of the functional \eqref{functional1}. In our numerical examples we consider  $f(x)=|x|$, $g(x)=|x|$ and  the data fidelity term is measured in the $L^{2}$  norm, i.e.,
\begin{equation}\label{functionaldenoising}
\overline{H_{ex}}(u)=\frac{1}{2}\int_{\Omega}(u_{0}-Tu)^{2}dx+\alpha\int_{\Omega}|\nabla u|~dx +\beta |D^{2}u|(\Omega).
\end{equation}
In order to discretise \eqref{functionaldenoising} we specify the corresponding discrete operators and norms that appear in the continuous functional.
 We  denote the discretised version of \eqref{functionaldenoising} with $J$.  In the discrete setting, $u$ is an element of $\mathbb{R}^{n\times m}$ and $T$ is a bounded linear operator from $\mathbb{R}^{n\times m}$ to $\mathbb{R}^{n\times m}$.  For $f(x)=|x|$ and $g(x)=|x|$, the discrete functional  $J$ is given by
\begin{equation}\label{functionaldiscrete}
J(u)=\frac{1}{2}\|u_{0}-Tu\|_{2}^{2}+\alpha \|\nabla u\|_{1}+\beta \|\nabla^{2} u\|_{1},
\end{equation}
where for every $u\in \mathbb{R}^{n\times m}$,  $v=(v_{1},v_{2})\in \left (\mathbb{R}^{n\times m}\right)^{2}$ and  $w=(w_{1},w_{2},w_{3},w_{4})\in \left (\mathbb{R}^{n\times m}\right)^{4}$ we define
\begin{align*}
\|u\|_{2}&:=\left (\sum_{i=1}^{n}\sum_{j=1}^{m} u(i,j)^{2}\right)^{1/2},\\
\|v\|_{2}&:=\left(\sum_{i=1}^{n}\sum_{j=1}^{m}v_{1}(i,j)^{2}+v_{2}(i,j)^{2}\right)^{1/2},\\
\|w\|_{2}&:=\Bigg(\sum_{i=1}^{n}\sum_{j=1}^{m}w_{1}(i,j)^{2}+w_{2}(i,j)^{2}+ w_{3}(i,j)^{2}+w_{4}(i,j)^{2}\Bigg)^{1/2},\\
\|v\|_{1}&:=\sum_{i=1}^{n}\sum_{j=1}^{m}\left (v_{1}(i,j)^{2}+v_{2}(i,j)^{2}\right)^{1/2},\\
\|w\|_{1}&:=\sum_{i=1}^{n}\sum_{j=1}^{m}\big(w_{1}(i,j)^{2}+w_{2}(i,j)^{2}+ w_{3}(i,j)^{2}+w_{4}(i,j)^{2}\big)^{1/2}.
\end{align*}
 For the  formulation of the discrete gradient and  Hessian operators we use  periodic boundary conditions and we follow \cite{wu2010augmented}. We  also refer the reader to \cite{mineipol} where  the form of the discrete differential operators is described in detail.  We define $\nabla$ and $\text{div}$ consistently with the continuous setting as adjoint operators and the same is done for  the Hessian $\nabla^{2}$ and its adjoint $\text{div}^{2}$.

In particular the first and second order divergence operators, $\text{div}$ and $\text{div}^{2}$, have the properties:
\[\text{div}:\left (\mathbb{R}^{n\times m}\right)^{2}\to \mathbb{R}^{n\times m} \quad \text{with}\quad-\text{div}(v)\cdot u=v\cdot \nabla u,\quad \forall u\in\mathbb{R}^{n\times m},\;v\in\left (\mathbb{R}^{n\times m}\right)^{2},\]
\[\text{div}^{2}:\left (\mathbb{R}^{n\times m}\right)^{4}\to \mathbb{R}^{n\times m}\quad \text{with}\quad \text{div}^{2}w\cdot u=w\cdot  \nabla^{2}u,\quad \forall u\in\mathbb{R}^{n\times m},\;w\in\left (\mathbb{R}^{n\times m}\right)^{4},\]
where the ``$\cdot$'' denotes the Euclidean inner product, if we consider $u$, $v$ and $w$ as large  vectors formed successively by the columns of the corresponding matrices.
\subsection{Constrained optimisation and the Bregman iteration}\label{seccon}
We now introduce the Bregman and the split Bregman iteration as a numerical method for the solution of \eqref{functionaldiscrete}. We would like to recall some basic aspects of the general theory of Bregman iteration before finishing this discussion with the convergence result in Theorem \ref{bregmanprop}. 

Suppose we have to solve the following constrained minimisation problem:
\begin{equation}\label{bregmancon}
\min_{u\in \mathbb{R}^{d}}E(u) \text{ such that } Au=b,
\end{equation}
 where the function $E$ is convex and $A$ is a linear map from $\mathbb{R}^{d}$ to $\mathbb{R}^{\ell}$. We transform the constrained minimisation problem (\ref{bregmancon}) into an unconstrained one, introducing a parameter $\lambda$: 
 \begin{equation} \label{bregmanuncon}
\sup_{\lambda}\min_{u\in \mathbb{R}^{d}} E(u)+\frac{\lambda}{2}\|Au-b\|_{2}^{2},
\end{equation}
where in order to satisfy the constraint $Au=b$ we have to let $\lambda$ increase to infinity.
Instead of doing that we perform the  Bregman iteration as it was proposed in \cite{OBG} and \cite{YOGD}:
\begin{center}
\line(1,0){230}
\end{center}

\begin{center}

\underline{\textbf{Bregman Iteration}  }\\
\end{center}
\begin{eqnarray}\label{miniter}
u_{k+1}&=&\underset{u\in\mathbb{R}^{d}}{\operatorname{argmin}}\;\;E(u)+\frac{\lambda}{2}\|Au-b_{k}\|_{2}^{2},\\
b_{k+1}&=&b_{k}+b-Au_{k+1}. \label{miniter2}
\end{eqnarray}
\begin{center}
\line(1,0){230}
\end{center}

In \cite{OBG}, assuming that \eqref{miniter} has a unique solution, the authors derive the following facts for the iterates $u_{k}$:
\begin{equation}\label{condition1}
\|Au_{k}-b\|_{2}\le \frac{M}{\sqrt{k-1}},\quad\text{for a fixed }M\ge0,\;k\ge2.
\end{equation}
\begin{equation}\label{condition2}
\sum_{k=1}^{\infty}\|Au_{k}-b\|_{2}^{2}<\infty,
\end{equation}
\begin{equation}\label{condition3}
\|Au_{k+1}-b\|_{2}^{2}\le \|Au_{k}-b\|_{2}^{2},\quad k\ge 1,
\end{equation}
and
\begin{equation}\label{condition4}
E(u_{k})<N\quad\text{for a constant }N\ge 0.
\end{equation}
The following theorem was proved in  \cite{goldstein2009split} in the case where the iterates of Bregman iteration satisfy the constraint in a finite number of iterations. In other words, it was proved that if for some iterate $u_{k_0}$ we have $Au_{k_{0}}=b$ then $u_{k_{0}}$ is a solution to the original constrained problem \eqref{bregmancon}. In the following theorem we give a  more general proof where we do not use that assumption, something that makes it a genuine contribution to the convergence theory of Bregman iteration.
\newtheorem{bregmanprop}[first]{Theorem}
\begin{bregmanprop}\label{bregmanprop}
Suppose that the constrained minimisation problem (\ref{bregmancon}) has a unique solution $u^{\ast}$. Moreover, suppose that the convex functional $E$ is  coercive and  \eqref{miniter} has a unique solution for every $k$. Then the sequence of the iterates $(u_{k})_{k\in \mathbb{N}}$ of Bregman iteration  converges to $u^{\ast}$.
\end{bregmanprop}
\begin{proof}
We have that the statements \eqref{condition1}-\eqref{condition4} hold.
Moreover, we have for every $k$,
\[b_{k}=b_{0}+\sum_{\nu=1}^{k}(b-Au_{\nu})\Rightarrow\]\[\|b_{k}\|_{2}\le\|b_{0}\|_{2}+\sum_{\nu=1}^{k} \|Au_{\nu}-b\|_{2}.\] 
From \eqref{condition4} and the coercivity of $E$ we get that the sequence $(\|u_{k}\|_{2})_{k\in\mathbb{N}}$ is bounded, say by a constant $C>0$. Thus, it suffices to show that every accumulation point of $(u_{k})_{k\in\mathbb{N}}$ is equal to $u^{\ast}$. Since $Au^{\ast}=b$, for every increasing sequence of naturals $(k_{\ell})_{\ell\in\mathbb{N}}$, we have that
\begin{eqnarray*}
\frac{\lambda}{2}\|Au^{\ast}-b_{k_{\ell}}\|_{2}^{2}&\le &\frac{\lambda}{2}\left(\|Au^{\ast}-Au_{k_{\ell}}\|_{2}+\|Au_{k_{\ell}}-b_{k_{\ell}}\|_{2} \right)^{2} \\
&=&\frac{\lambda}{2}\left ( \|Au_{k_{\ell}}-b\|_{2}+ \|Au_{k_{\ell}}-b_{k_{\ell}}\|_{2}\right )^{2}\\
							&\le&\frac{\lambda M^{2}}{2(k_{\ell}-1)}+\lambda\|Au_{k_{\ell}}-b\|_{2}\|Au_{k_{\ell}}-b_{k_{\ell}}\|_{2}+\frac{\lambda}{2}\|Au_{k_{\ell}}-b_{k_{\ell}}\|_{2}^{2}.
\end{eqnarray*}
Now because of the fact that 
\[E(u_{k_{\ell}})+\frac{\lambda}{2}\|Au_{k_{\ell}}-b_{k_{\ell}-1}\|_{2}^{2}\le E(u^{\ast})+\frac{\lambda}{2}\|Au^{\ast}-b_{k_{\ell}-1}\|_{2}^{2},\]
we get that
\begin{eqnarray}\label{bigineq}
E(u_{k_{\ell}})&\le& E(u^{\ast})+\frac{\lambda M^{2}}{2(k_{\ell}-1)}+\lambda\|Au_{k_{\ell}}-b\|_{2}\|Au_{k_{\ell}}-b_{k_{\ell}-1}\|_{2}\nonumber\\
			&\le&E(u^{\ast})+\frac{\lambda M^{2}}{2(k_{\ell}-1)}+\lambda\|Au_{k_{\ell}}-b\|_{2}\|Au_{k_{\ell}}\|_{2}+\lambda\|Au_{k_{\ell}}-b\|_{2}\|b_{k_{\ell}-1}\|_{2}\nonumber\\
			&\le& E(u^{\ast})+\frac{\lambda M^{2}}{2(k_{\ell}-1)}+\frac{\lambda M \|A\|\|u_{k_{\ell}}\|_{2}}{\sqrt{k_{\ell}-1}}+\lambda\|Au_{k_{\ell}}-b\|_{2}\|b_{k_{\ell}-1}\|_{2}\nonumber\\
			&\le&E(u^{\ast})+\frac{\lambda M^{2}}{2(k_{\ell}-1)}+\frac{\lambda M\|A\|C}{\sqrt{k_{\ell}-1}}+\lambda\|Au_{k_{\ell}}-b\|_{2}\left (\|b_{0}\|_{2}+\sum_{\nu=1}^{k_{\ell}}\|Au_{\nu}-b\|_{2}\right )\nonumber\\
			&\le&E(u^{\ast})+\frac{\lambda M^{2}}{2(k_{\ell}-1)}+\frac{\lambda M\|A\|C}{\sqrt{k_{\ell}-1}}
			+\frac{\lambda M\|b_{0}\|_{2}}{\sqrt{k_{\ell}-1}}
			+\lambda \|Au_{k_{\ell}}-b\|_{2}\sum_{\nu=1}^{k_{\ell}}\|Au_{\nu}-b\|_{2}.
\end{eqnarray}
 Suppose now that $(u_{k_{\ell}})_{\ell\in\mathbb{N}}$ converges to some $\tilde{u}$ as $\ell$ goes to infinity. Then taking into account \eqref{condition1}, $\tilde{u}$ also satisfies 
 \begin{equation}\label{whatiwant1}
 \|A\tilde{u}-b\|_{2}=0.
 \end{equation}
 Taking limits in \eqref{bigineq} and using Kronecker's lemma, see Lemma \ref{kronecker}, we have that the limit in the right hand side of \eqref{bigineq} is $E(u^{\ast})$. Thus, we have
 \[E(\tilde{u})\le E(u^{\ast}),\]
 and since $u^{\ast}$ is the solution of the constrained minimisation problem and \eqref{whatiwant1} holds, we have $\tilde{u}=u^{\ast}$. Since every accumulation point of the bounded sequence $(u_{k})_{k\in\mathbb{N}}$ is equal to $u^{\ast}$, we conclude that the whole sequence  converges to $u^{\ast}$.
\end{proof}
For more information about the use of Bregman iteration in $L^{1}$ regularised problems we refer the reader to \cite{goldstein2009split,OBG,YOGD}.  

\subsection{Numerical solution of our minimisation problem}\label{Bregmanourcase}
In this section, we explain how the  Bregman iteration \eqref{miniter}-\eqref{miniter2} together with an operator splitting technique can be used to implement numerically the minimisation of  functional \eqref{functionaldiscrete}. The idea originates from \cite{goldstein2009split} where such a procedure is applied to total variation minimisation and is given the name split Bregman algorithm. This iterative technique is equivalent to certain instances of combinations of the augmented Lagrangian method with classical operator splitting such as Douglas-Rachford, see \cite{setzer}. We also refer the reader to the paper of Benning, Brune, Burger and  M\"uller \cite{TGVbregman}, for applications of Bregman methods to higher-order regularisation models for image reconstruction.

Exemplarily, we present the resulting algorithm for the minimisation of J in \eqref{functionaldiscrete}, i.e., for $f(x)=|x|$, $g(x)=|x|$ and $L^2$ data fidelity term. 
Recall that we want to solve the following unconstrained minimisation problem:
\begin{equation}\label{minimisationstep1}
\min_{u\in \mathbb{R}^{n\times m}} \frac{1}{2}\|u_{0}-Tu\|_{2}^{2}+\alpha \|\nabla u\|_{1}+\beta \|\nabla^{2} u\|_{1}.
\end{equation}
The derivation of the split Bregman algorithm for solving \eqref{minimisationstep1} starts with the first  observation that the above minimisation problem  is equivalent to the following constrained minimisation  problem:
\begin{equation}\label{minimisationstep2}
\min_{\substack{u\in \mathbb{R}^{n\times m}  \\ v \in \left (\mathbb{R}^{n\times m}\right)^{2}  \\ w \in \left (\mathbb{R}^{n\times m}\right)^{4}}  }  \frac{1}{2}\|u_{0}-Tu\|_{2}^{2}+\alpha \|v\|_{1}+\beta \|w\|_{1}, \text{ such that } v=\nabla u,\; w= \nabla^{2}u.
\end{equation}
It is clear that since the gradient and the Hessian are linear operations,  the minimisation problem  (\ref{minimisationstep2}) can be reformulated  into the more general problem:
\begin{equation}\label{minimisationstep3}
\min_{\omega\in \mathbb{R}^{d}} E(\omega) \text{ such  that }A\omega=b,
\end{equation}
where $E:\mathbb{R}^{d}\to\mathbb{R^{+}}$ is convex, $A$ is a $d\times d$ matrix and $b$ is a vector of length $d=7nm$. 

It is easy to see that the iterative scheme of the type (\ref{miniter})-(\ref{miniter2}) that corresponds to the constraint minimisation problem (\ref{minimisationstep2}) is :

\begin{eqnarray}
(u^{k+1},v^{k+1},w^{k+1})&=& \underset{u,v,w}{\operatorname{argmin}}\;\;\frac{1}{2}\|u_{0}-Tu\|_{2}^{2}+\alpha \|v\|_{1}+\beta \|w\|_{1}+\frac{\lambda}{2}\|b_{1}^{k}+\nabla u-v\|_{2}^{2}\label{oursminiter1}\\&&+\frac{\lambda}{2}\|b_{2}^{k}+\nabla^{2}u-w\|_{2}^{2},\nonumber\\
b_{1}^{k+1}&=& b_{1}^{k}+\nabla u^{k+1}-v^{k+1}\label{oursminiter2},\\
 b_{2}^{k+1}&=&b_{2}^{k}+\nabla ^{2}u^{k+1}-w^{k+1}.
\label{oursminiter3}
\end{eqnarray}
where $b_{1}^{k+1}=(b_{1,1}^{k+1},b_{1,2}^{k+1})\in(\mathbb{R}^{n\times m})^{2}$ and $b_{2}^{k+1}=(b_{2,11}^{k+1},b_{2,22}^{k+1},b_{2,12}^{k+1},b_{2,21}^{k+1})\in(\mathbb{R}^{n\times m})^{4}$.

\newtheorem{rem}[first]{Remark}
\begin{rem}\label{rem}
Notice that at least in the case where $T$ is injective (denoising, deblurring), the minimisation problem (\ref{oursminiter1}) has a unique solution. Moreover in that case, the functional $E$ is  coercive and the constrained minimisation problem (\ref{minimisationstep2}) has a unique solution. Thus,  Theorem \ref{bregmanprop} holds.
\end{rem}

Our next concern is the efficient numerical solution of the minimisation  problem (\ref{oursminiter1}).
We follow \cite{goldstein2009split} and iteratively minimise with
respect to $u$, $v$ and $w$ alternatingly:
\begin{center}
\line(1,0){230}
\end{center}

\begin{center}

\underline{\textbf{Split Bregman for } $\mathbf{TV-TV^{2}-L^2}$ }\\
\end{center}

\begin{eqnarray}
u^{k+1}&=&\underset{u\in \mathbb{R}^{n\times
m}}{\operatorname{argmin}}\;\;
\frac{1}{2}\|u_{0}-Tu\|_{2}^{2}+\frac{\lambda}{2}\|b_{1}^{k}+\nabla u-v^{k}\|_{2}^{2}+\frac{\lambda}{2}\|b_{2}^{k}+\nabla^{2}u-w^{k}\|_{2}^{2},\label{step1}\\
v^{k+1}&=&\underset{v\in \left(\mathbb{R}^{n\times
m}\right)^{2}}{\operatorname{argmin}}\;\;\alpha\|v\|_{1}+\frac{\lambda}{2}\|b_{1}^{k}+\nabla u^{k+1}-v\|_{2}^{2},\label{step2}\\
w^{k+1}&=&\underset{w\in \left(\mathbb{R}^{n\times
m}\right)^{4}}{\operatorname{argmin}}\;\;\beta\|w\|_{1}+\frac{\lambda}{2}\|b_{2}^{k}+\nabla^{2}u^{k+1}-w\|_{2}^{2},\label{step3}\\
b_{1}^{k+1}&=& b_{1}^{k}+\nabla u^{k+1}-v^{k+1},\label{step4}\\
 b_{2}^{k+1}&=&b_{2}^{k}+\nabla ^{2}u^{k+1}-w^{k+1}.\label{step5}
\end{eqnarray}
\begin{center}
\line(1,0){230}
\end{center}
The above alternating minimisation scheme, make up the split Bregman iteration that is proposed in \cite{goldstein2009split} to solve the total variation minimisation problem as well as problems related to compressed sensing. For convergence properties of the split Bregman iteration and also other splitting techniques we refer the reader to \cite{combettes2006signal,esser2009general,setzer}.  In \cite{wu2010augmented} and \cite{YOGD}  it is noted that the Bregman iteration coincides with the augmented Lagrangian method. Minimising alternatingly with respect to the variables in the augmented Lagrangian method results to the alternating direction method of multipliers (ADMM), see \cite{Gabay1983299}. Thus split Bregman is equivalent to ADMM. In  \cite{Eckstein} and \cite{Gabay1983299}  it is shown that ADMM is equivalent to the  Douglas-Rachford splitting algorithm and thus convergence is guaranteed. We refer the reader to \cite{setzer} for an interesting study in 
this subject.

We now discuss how we solve each of the minimisation problems \eqref{step1}-\eqref{step3}.
The problem (\ref{step1}) is quadratic and it is solved through its optimality condition. This condition  reads:
\begin{eqnarray}\label{optcont}
&&T^{\ast}Tu-\lambda\text{div}\left (\nabla u\right)+\lambda \text{div}^{2}\left (\nabla ^{2} u\right)=T^{\ast}(u_{0})+\lambda \text{div}\left((b_{1}^{k}-v^{k})\right)-\lambda\text{div}^{2}\left((b_{2}^{k}-w^{k})\right),
\end{eqnarray}
where $T^{\ast}$ denotes the adjoint of the discrete operator $T$. Since all the operators that appear in (\ref{optcont}) are linear, this condition leads to a linear system of equations with $nm$ unknowns.  In \cite{goldstein2009split}, one iteration of the Gauss-Seidel method is used to approximate the solution of the corresponding optimality condition of \eqref{optcont}. However, numerical experiments have shown that in the higher-order case it is preferable and more robust  to solve this problem exactly.  Since we impose periodic boundary conditions for the discrete differential operators, this can be done efficiently using fast Fourier transform, see \cite{mineipol,wu2010augmented}.

The solutions of the minimisation problems (\ref{step2}) and (\ref{step3}) can be obtained exactly through a  generalised shrinkage method as it was done in \cite{goldstein2009split} and \cite{Wang_Yin_2007}. It is a simple computation to check that if $a\in\mathbb{R}^{n}$ then the solution to the  problem
\begin{equation}\label{basicshr}
\min_{x\in \mathbb{R}^{n}}\|x\|_{2}+\frac{\lambda}{2}\|x-a\|_{2}^{2},
\end{equation}
can be obtained through the following formula:
\begin{equation}\label{shformula}
x=\mathbb{S}_{\frac{1}{\lambda}}(a):=\max\left(\|a\|_{2}-\frac{1}{\lambda},0\right)\frac{a}{\|a\|_{2}}.
\end{equation}
Each of the objective functionals in (\ref{step2}) and (\ref{step3}) can be written as a sum of functionals of the same type in (\ref{basicshr}) where $n=2,\,4$ respectively. Thus the solution to the problems (\ref{step2}) and (\ref{step3}) can be computed as follows:
\begin{eqnarray}\label{s2}
v^{k+1}(i,j)&=&\left(v_{1}^{k+1}(i,j),v_{2}^{k+1}(i,j)\right)\nonumber\\&=&\mathbb{S}_{\frac{\alpha}{\lambda}}\left(b_{1}^{k}(i,j)+\nabla u^{k+1}(i,j)\right),
\end{eqnarray}
\begin{eqnarray}\label{s2}
w^{k+1}(i,j)&=&\big(w_{1}^{k+1}(i,j),w_{2}^{k+1}(i,j), w_{3}^{k+1}(i,j),w_{4}^{k+1}(i,j)\big)\nonumber\\&=&\mathbb{S}_{\frac{\alpha}{\lambda}}\left(b_{2}^{k}(i,j)+\nabla^{2} u^{k+1}(i,j)\right).
\end{eqnarray}
for $i=1,\ldots,n$ and $j=1,\ldots,m$.

Let us note that the algorithm \eqref{step1}-\eqref{step5} can be easily generalised to colour images, again see \cite{mineipol,wu2010augmented}.\\

We have performed numerical experiments for image denoising, deblurring and inpainting using the algorithm \eqref{step1}-\eqref{step5}.  In all of our numerical examples the range of image values is $[0,1]$ (zero for black  and one for white).

Notice that different values of $\lambda$ can result in different speeds of convergence. Also, one can consider having two parameters $\lambda_{1}$ and $\lambda_{2}$ for the first and the second order term respectively. We can easily check that the Bregman iteration converges in this case as well. Even though it is not obvious a priori how to choose $\lambda_{1}$ and $\lambda_{2}$ in order to have fast convergence, this choice only has to be done once and a potential user does not have to worry about them.  Empirically we have found that $\lambda_{1}$ and $\lambda_{2}$ have to be a few orders of magnitude larger than $\alpha$ and $\beta$ respectively. We have found that a good empirical rule is to choose $\lambda_{1}=10\alpha$ and $\lambda_{2}=10\beta$ or even $\lambda_{1}=100\alpha$ and $\lambda_{2}=100\beta$. In \cite{Boyd} there is an interesting discussion about that matter. See also \cite{mineipol}, where it is shown with numerical experiments for the case of inpainting, how this choice of $\lambda$'s
 results in different speed of convergence and different behaviour of the intermediate iterates.
 
\section{Applications in Denoising}\label{denoising}

In this section we discuss the application of the TV-TV$^2$ approach \eqref{functionaldiscrete} to image denoising, where the operator $T$  equals the identity.
We have performed experiments to images that have been corrupted with Gaussian noise, thus the $L^{2}$ norm  in the fidelity term is the most suitable. We compare our method with infimal convolution \cite{ChambolleLions} solved also with a  split Bregman scheme and the total generalised variation \cite{TGV} solved  with the primal-dual method of Chambolle and Pock \cite{chambolle2011first} as it is described in \cite{tgvcolour}.
We present examples of \eqref{functionaldiscrete} for $\alpha=0$, $\beta\ne 0$ and $\alpha\ne 0$, $\beta=0$. Note that for $\beta=0$, our model corresponds to the classical Rudin-Osher-Fatemi (ROF) denoising model, while for $\alpha=0$ it corresponds to the pure TV$^{2}$ restoration \cite{Piffet}.  Our basic synthetic test image is shown in Figure \ref{mytestimage}.
\begin{figure}[ht]
\begin{center}
\includegraphics[height=5cm]{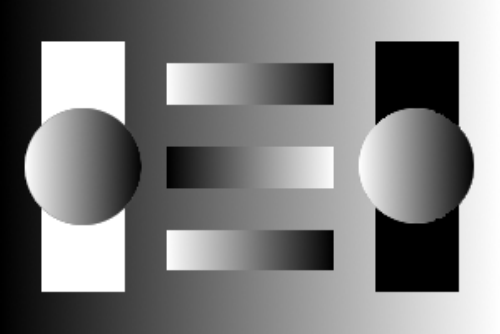}
\caption{Main test image. Resolution: 200$\times$ 300 pixels} 
\label{mytestimage}
\end{center}
\end{figure}

Our main assessment for the quality of the reconstruction is the \emph{structural similarity index} SSIM \cite{wang2009mean,wang2004image}. The reason for that choice is that in contrast to traditional quality measures like the peak signal-to-noise ratio PSNR, the SSIM index also assesses the conservation of the structural information of the reconstructed image. Note that a perfect reconstruction has SSIM value equal to 1. A justification for the choice of SSIM as a good fidelity measure instead of the traditional PSNR can be seen in Figure \ref{just}. The second and the third image are denoising results with the first order method ($\beta=0$,  Gaussian noise, Variance = 0.5). The middle picture is the one with the highest SSIM value (0.6595) while the SSIM value of the right picture is significantly lower (0.4955). This assessment  comes into an agreement with the human point of view since, even though this is subjective, one would consider the middle picture as a better reconstruction.  On the other hand the middle picture has slightly smaller PSNR value (14.02) than the right one (14.63), which was the highest PSNR value. Similar results are obtained for $\beta\ne 0$.
\begin{figure*}[ht]
\begin{center}
\includegraphics[scale=0.63]{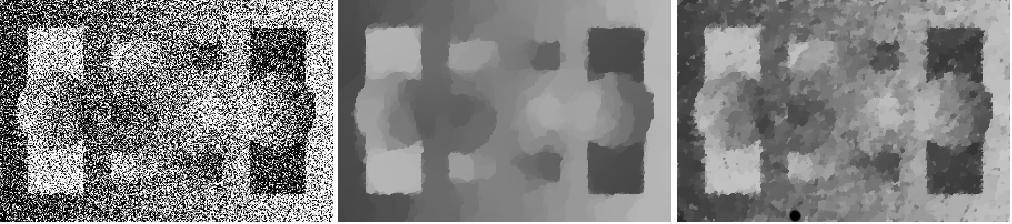}
\end{center}
\caption{Justification for the usage of SSIM index. The initial image, seen on Figure \ref{mytestimage}, is contaminated with Gaussian noise of variance 0.5 (left). We provide the  best SSIM valued (middle) (SSIM=0.6595, PSNR=14.02) and the  best PSNR valued (right) (SSIM=0.4955, PSNR=14.63) reconstructions among reconstructed images with the first order method ($\beta=0$).  The better SSIM assessment of the first image agrees more with the human perception}
\label{just}
\end{figure*}

As a stopping criterium for our algorithm we use a predefined  number of iterations. In most examples this number is 300.  We observe that after 80-100 iterations the relative residual  of the iterates is of the order of $10^{-3}$ or lower (see also Table \ref{times}) and hence no noticeable change in the iterates is observed after that.

In the following we shall examine whether the introduction of the higher-order term ($\beta\ne 0$) in the denoising procedure, produces results of higher quality.

\begin{figure*}[ht]
\begin{center}
\subfigure[Clean image, SSIM=1]{
\includegraphics[scale=0.63]{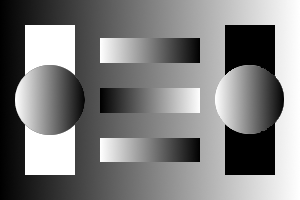}
}
\subfigure[Noisy image, Gaussian noise, variance=0.005, SSIM=0.3261]{
\includegraphics[scale=0.63]{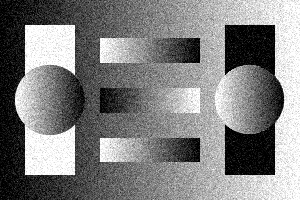}
}
\subfigure[TV denoising, $\alpha$$=$$0.12$, SSIM=0.8979]{
\includegraphics[scale=0.63]{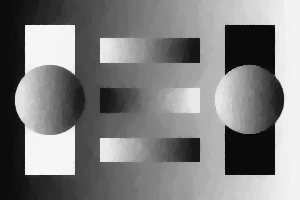}
}
\subfigure[TGV denoising, SSIM=0.9249]{
\includegraphics[scale=0.63]{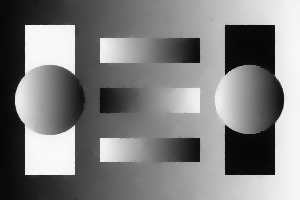}
}
\subfigure[Inf-convolution  denoising, SSIM=\newline0.9053]{
\includegraphics[scale=0.63]{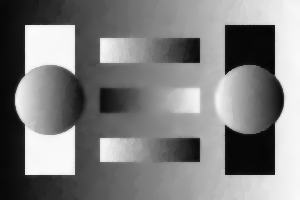}
}
\subfigure[TV-TV$^2$ denoising, $\alpha$$=$$0.06$, $\beta$$=$$0.03$, SSIM=0.9081]{
\includegraphics[scale=0.63]{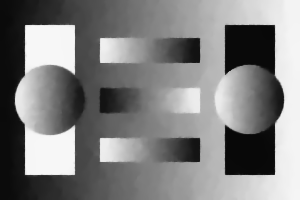}
}
\subfigure[TV$^2$ denoising,  $\beta$$=$$0.07$, SSIM$=$ \newline$0.8988$]{
\includegraphics[scale=0.63]{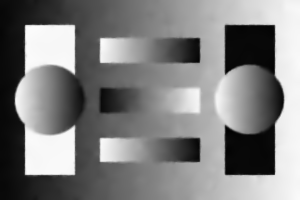}
}
\subfigure[TV-TV$^2$ denoising, $\alpha$$=$$0.06$, $\beta$$=$$0.06$, SSIM=0.8989]{
\includegraphics[scale=0.63]{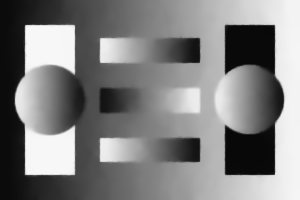}
}
\subfigure[TV-TV$^2$ denoising, $\alpha$$=$$0.06$, $\beta$$=$$0.06$, SSIM=0.9463, Post-processing: GIMP sharpening \& contrast]{
\includegraphics[scale=0.63]{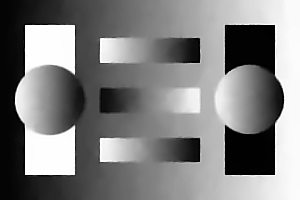}
}
\caption{Denoising of a synthetic image that has been corrupted with Gaussian noise of variance 0.005. We chose $\lambda_{1}=\lambda_{2}=1$ for these implementations }
\label{denoisingimages}
\end{center}
\end{figure*}

\begin{figure*}[ht]
\begin{center}
\subfigure[Clean image, SSIM=1]{
\includegraphics[scale=0.3]{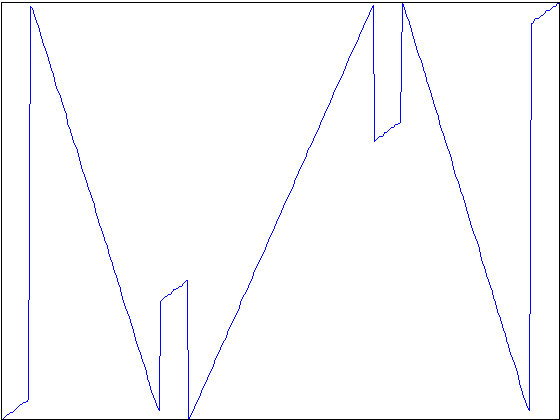}
}
\subfigure[Noisy image, Gaussian noise, variance=0.005, SSIM=0.3261]{
\includegraphics[scale=0.3]{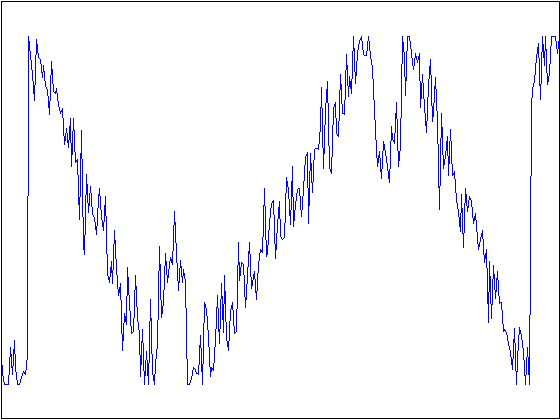}
}
\subfigure[TV denoising, $\alpha$$=$$0.12$,  SSIM=\newline0.8979]{
\includegraphics[scale=0.3]{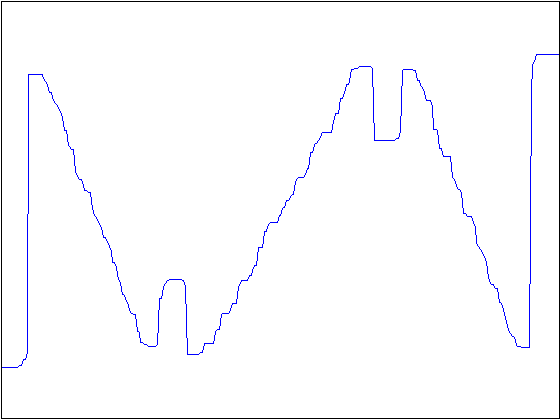}
}
\subfigure[TGV denoising, SSIM=0.9249]{
\includegraphics[scale=0.3]{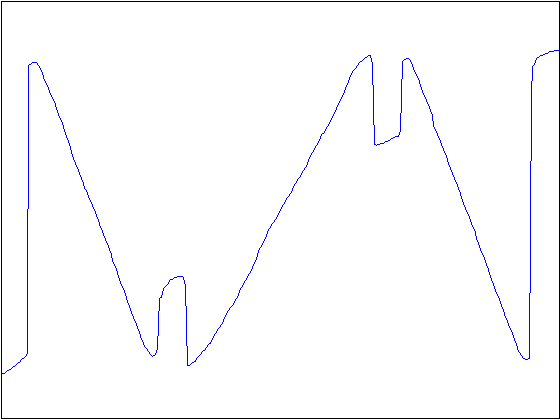}
}
\subfigure[Inf-convolution  denoising,\newline SSIM=0.9053]{
\includegraphics[scale=0.3]{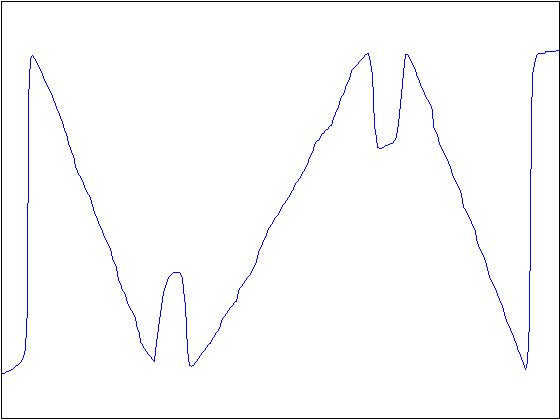}
}
\subfigure[TV-TV$^2$ denoising, $\alpha$$=$$0.06$, \newline$\beta$$=$$0.03$, SSIM=0.9081]{
\includegraphics[scale=0.3]{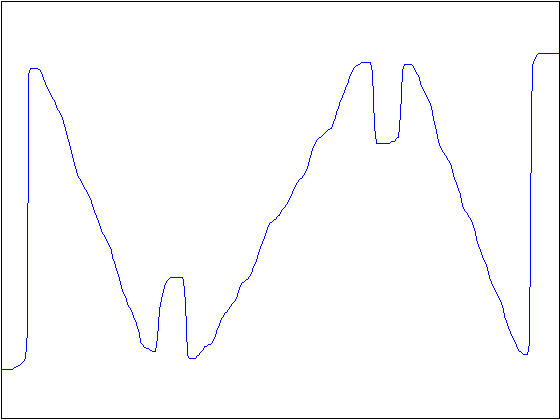}
}
\subfigure[TV$^2$ denoising,  $\beta$$=$$0.07$, SSIM$=$ \newline$0.8988$]{
\includegraphics[scale=0.3]{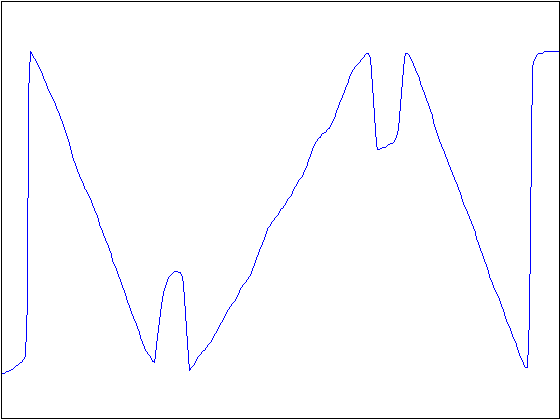}
}
\subfigure[TV-TV$^2$ denoising, $\alpha$$=$$0.06$,\newline $\beta$$=$$0.06$, SSIM=0.8989]{
\includegraphics[scale=0.3]{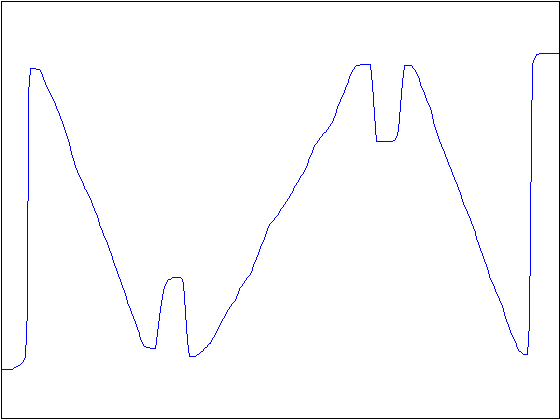}
}
\subfigure[TV-TV$^2$ denoising, $\alpha$$=$$0.06$, $\beta$$=$$0.06$, SSIM=0.9463, \newline Post-processing: GIMP sharpening \& contrast]{
\includegraphics[scale=0.3]{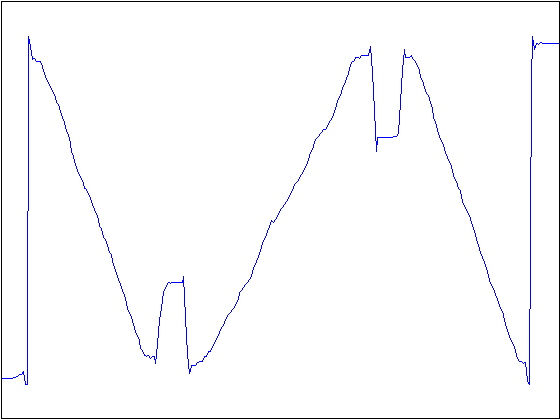}
}
\caption{Corresponding middle row slices of images in Figure \ref{denoisingimages}}
\label{denoisingslices}
\end{center}
\end{figure*}

\begin{figure*}[ht]
\begin{center}
\includegraphics[scale=0.6 ]{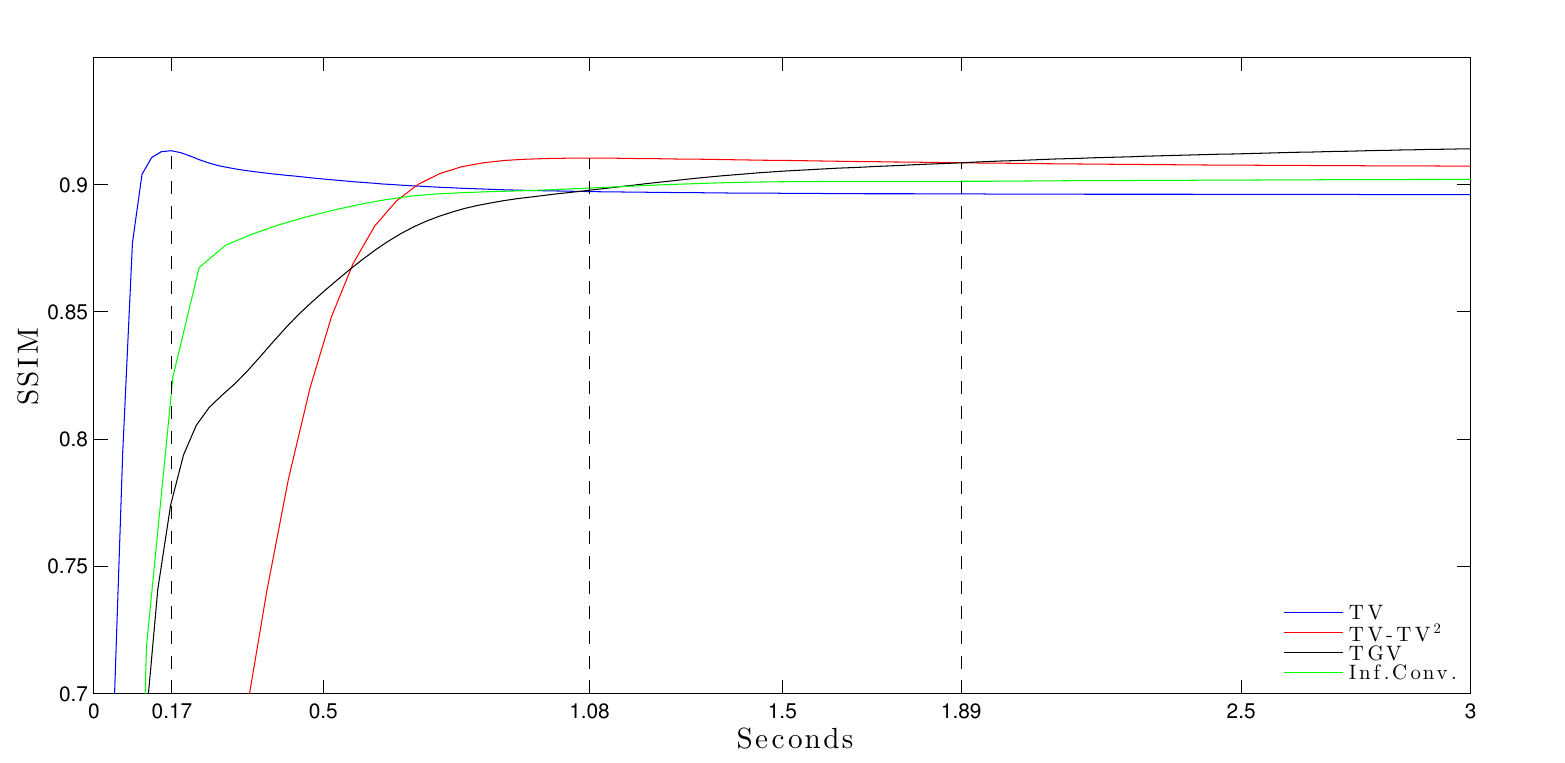}
\caption{Evolution of the SSIM index with absolute CPU time for the examples of Figure \ref{denoisingimages}. For TV denoising the SSIM value peaks after 0.17 seconds (0.9130) and the after it drops sharply when the staircasing appears, see corresponding comments on \cite{goldstein2009split}. For TV-TV$^{2}$ the peak appears after 1.08 seconds (0.9103) and remains essentially constant. The TGV iteration starts to outperform the methods after 1.89 seconds. This shows the potential of split Bregman to produce visually satisfactory results before convergence has occurred, in contrast with the primal-dual method}
\label{ssimevolution}
\end{center}
\end{figure*}

\begin{figure*}[ht]
\begin{center}
\includegraphics[scale=0.45]{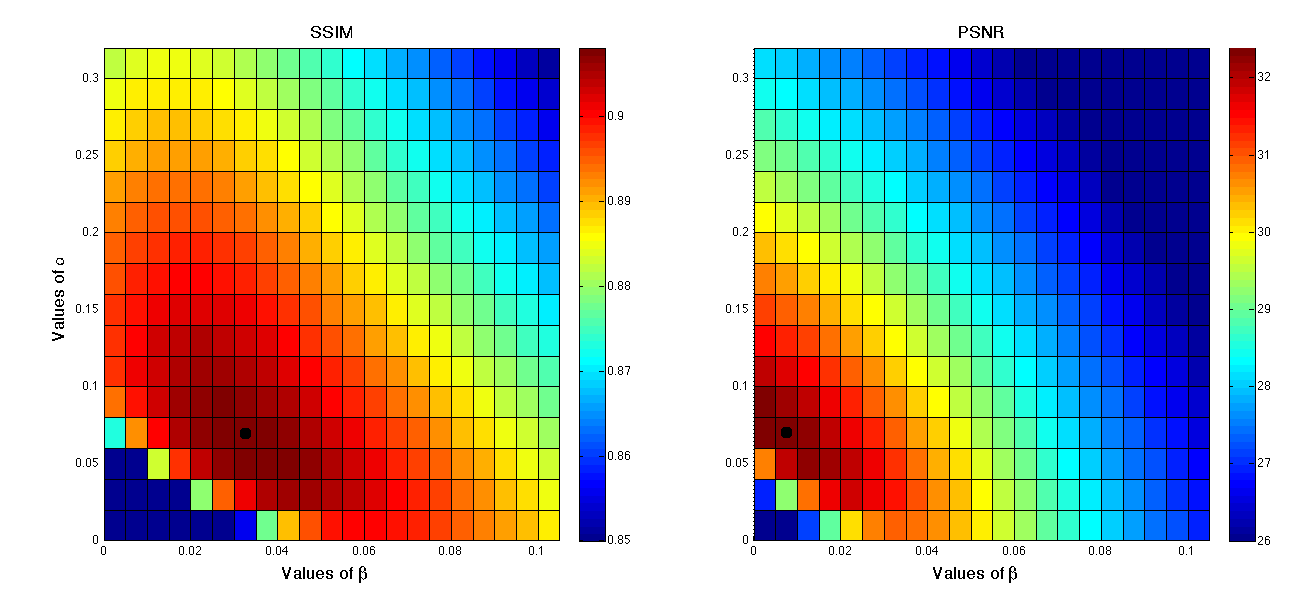}
\end{center}
\caption{Plot of the SSIM and PSNR values of the restored image as functions of $\alpha$ and $\beta$. For display convenience all the values under 0.85 (SSIM) and 26 (PSNR) were coloured with dark blue. The dotted cells corresponds to the highest SSIM (0.9081) and PSNR (32.39) value that were achieved for $\alpha=0.06$, $\beta=0.03$ and $\alpha=0.06$, $\beta=0.005$ respectively. Note that the first column in both plots corresponds to TV denoising, ($\beta=0$). The original image was corrupted with Gaussian noise of variance $0.005$}
\label{hroma}
\end{figure*}

\begin{figure*}[ht]
\begin{center}
\includegraphics[scale=0.5]{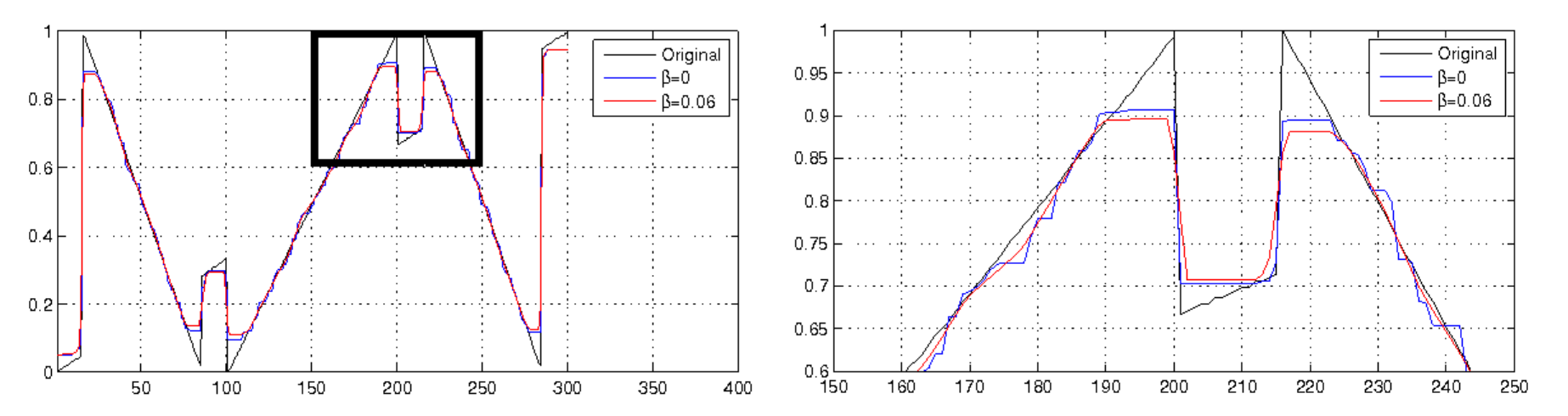}
\end{center}
\caption{Left: Middle row slices of reconstructed images with $\alpha=0.12$, $\beta=0$ (blue colour) and $\alpha=0.12$. $\beta=0.06$ (red colour). Slices of the original image are plotted with black colour. Right: Detail of the first plot. Even though the higher-order method eliminates the staircasing effect, it also results to further slight loss of contrast}
\label{contrast}
\end{figure*}
 The noise has been produced with
MATLAB's built in function \emph{imnoise}. 

\begin{figure*}[ht]
\begin{center}
\subfigure[Clean image, SSIM=1]{
\includegraphics[scale=0.45]{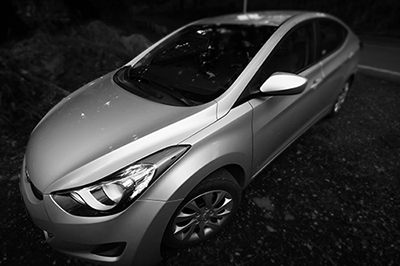}
}
\subfigure[Noisy image, Gaussian noise, variance=0.005, SSIM=0.4436]{
\includegraphics[scale=0.45]{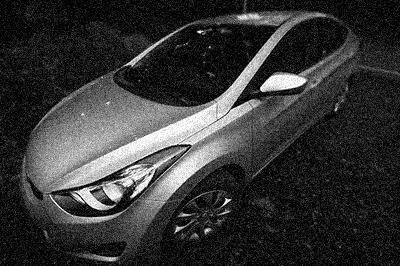}
}
\subfigure[TV denoising, $\alpha=0.05$, \newline SSIM=0.8168]{
\includegraphics[scale=0.45]{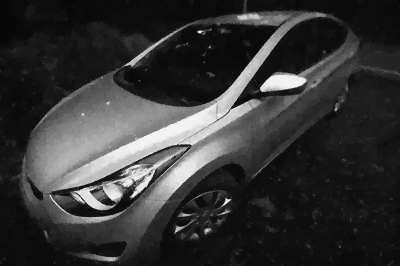}
}

\subfigure[TV-TV$^{2}$ denoising, $\alpha=0.017$, \newline$\beta=0.017$, SSIM=0.8319]{
\includegraphics[scale=0.45]{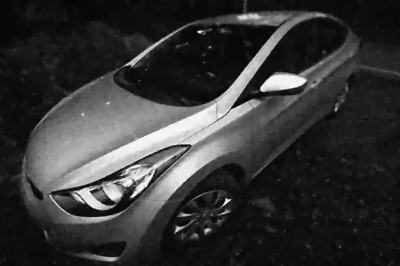}
}
\subfigure[TV-TV$^{2}$ denoising, $\alpha=0.023$,\newline $\beta=0.023$, SSIM=0.8185]{
\includegraphics[scale=0.45]{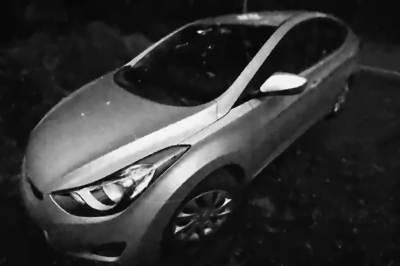}
}
\subfigure[TV$^{2}$ denoising, $\beta=0.05$, \newline SSIM=0.8171]{
\includegraphics[scale=0.45]{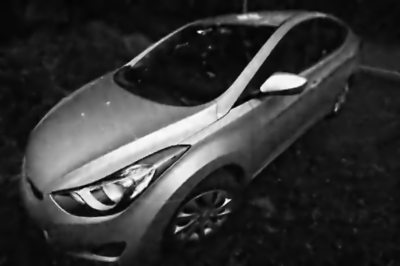}
}

\subfigure[Clean image, SSIM=1]{
\includegraphics[scale=1.2]{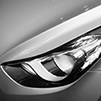}
}
\subfigure[Noisy image, Gaussian noise, variance=0.005,\newline SSIM=0.4436]{
\includegraphics[scale=1.2]{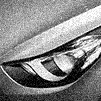}
}
\subfigure[TV denoising,\newline $\alpha=0.05$,  SSIM=0.8168]{
\includegraphics[scale=1.2]{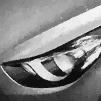}
}

\subfigure[TV-TV$^{2}$ denoising,\newline $\alpha=0.017$, $\beta=0.017$,\newline SSIM=0.8319]{
\includegraphics[scale=1.2]{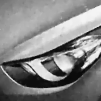}
}
\subfigure[TV-TV$^{2}$ denoising,\newline $\alpha=0.023$, $\beta=0.023$,\newline SSIM=0.8185]{
\includegraphics[scale=1.2]{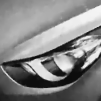}
}
\subfigure[TV$^{2}$ denoising, \newline $\beta=0.05$, SSIM=0.8171]{
\includegraphics[scale=1.2]{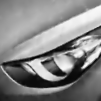}
}
\caption{Denoising of a natural image that has been corrupted with Gaussian noise of variance 0.005. We chose $\lambda_{1}=\lambda_{2}=1$ for these implementations}
\label{denoisingnatural}
\end{center}
\end{figure*}

Figure \ref{denoisingimages} depicts one denoising example, where the original image is corrupted with Gaussian noise of variance 0.005. For better visualisation, we  include the middle row slices of all the reconstructions in Figure \ref{denoisingslices}. The highest SSIM value for TV denoising is achieved for $\alpha=0.12$ (SSIM=0.8979) while the highest one for TV-TV$^{2}$  is achieved for $\alpha=0.06$, $\beta=0.03$ (SSIM=0.9081). This is slightly better than infimal convolution (SSIM=0.9053).  Note, however, that this optimal combination of $\alpha$ and $\beta$  in terms of SSIM   does not always correspond to best visual result. In general, the latter corresponds  to a slightly bigger $\beta$ than the one chosen by SSIM, see Figure \ref{denoisingimages}(h). Still, for proof of concept, we prefer to stick with an objective quality measure and SSIM, in our opinion, is the most reliable choice for that matter. In the image of Figure \ref{denoisingimages}(h) the staircasing effect has almost disappeared and the image is still \emph{pleasant} to the human eye despite being slightly more blurry. This slight blur, which is the price that the method pays for the removal of the staircasing effect, can be easily and efficiently removed in post-processing using simple sharpening filters, e.g. in GIMP. We did that in Figure \ref{denoisingimages}(i), also adjusting the contrast, achieving a very good result both visually and SSIM-wise (0.9463).
 
The highest SSIM value is achieved by TGV (0.9249), delivering a reconstruction of very good quality. However, TGV converges slowly to the true solution. In order to check that, we compute the ground truth solution (denoted by $GT$) for the parameters of the TGV problem that correspond to Figure \ref{denoisingimages}(d), by taking a large amount of iterations (2000). We check the GPU time that is needed for the iterates to have a relative residual 
\[\frac{\|u^{k}-GT\|_{2}}{\|GT\|_{2}}\le  10^{-3}\]
and we do the same for the TV-TV$^{2}$  example of Figure \ref{denoisingimages}(f). For TGV, it takes 1297 iterations (primal-dual method \cite{chambolle2011first}) and 36.25 seconds while for TV-TV$^{2}$ it takes 86  split Bregman iterations and 4.05 seconds, see Table \ref{times}.  That makes our method more suitable for cases where fast but not necessarily optimal results are needed, e.g. video processing.

In order to examine the quality of the reconstructions that are produced from each method as the number of iteration increases we have plotted in Figure \ref{ssimevolution} the evolution of SSIM values. In the horizontal axis, instead of the number of iterations we put the absolute CPU time calculated by the product of the number of iterations times the CPU time per iteration as it is seen in Table \ref{times}. We observe that, for TGV, the SSIM value increases gradually with time, while for the methods solved with split Bregman the image quality peaks very quickly and then remains almost constant except for TV where the staircasing appears in the later iterations.

Next we check how the SSIM and PSNR values of the restored images behave as a function of the weighting parameters $\alpha$ and $\beta$. In Figure \ref{hroma} we plot the results for $\alpha=0,\,0.02,\, 0.04,\ldots,\, 0.3$ and $\beta=0,\,0.005,\,0.01,\ldots,\, 0.1$. The plots suggest that both quality measures behave in a continuous way and that they have a global maximum. However, PSNR tends to rate higher those images that have been processed with a small value of $\beta$ or even with $\beta=0$ which is not the case for SSIM. An explanation for that is that higher values of $\beta$ result to a further loss of contrast, see Figure \ref{contrast}, something that is penalised by the PSNR. The SSIM index penalises the loss of contrast as well but it also penalises the creation of the staircasing effect. This is another indication for the suitability of SSIM over PNSR. Note also, that the contrast can be recovered easily in a post-processing stage while it is not an easy task to reduce the staircasing effect using conventional processing tools.

Finally, in Figure \ref{denoisingnatural} we perform denoising in an natural image which has been corrupted with Gaussian noise also of variance $0.005$. The staircasing of TV denoising (SSIM=0.8168) is obvious in Figures \ref{denoisingnatural}(c) and (i). The overall best performance of the TV-TV$^{2}$ method (SSIM=0.8319) is achieved by choosing $\alpha=\beta=0.017$, Figure \ref{denoisingnatural}(d). However, one can get satisfactory results by choosing $\alpha=\beta=0.023$ (SSIM=0.8185), eliminating further the staircasing without blurring the image too much, compare for example the details in the Figures \ref{denoisingnatural}(j) and (k).

\section{Applications in deblurring}\label{deblurring}
\begin{figure*}[ht]
\subfigure[Clean image , SSIM=1]{
\includegraphics[scale=0.63]{fig8.png}
}
\subfigure[Blurred and noisy image, SSIM=\newline0.8003]{
\includegraphics[scale=0.63]{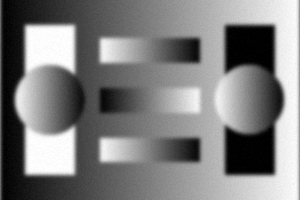}
}
\subfigure[TV deblurring, $\alpha$$=$$0.006$, SSIM=\newline0.9680]{
\includegraphics[scale=0.63]{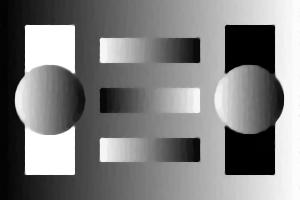}
}
\subfigure[TGV deblurring, SSIM=0.9806]{
\includegraphics[scale=0.5035]{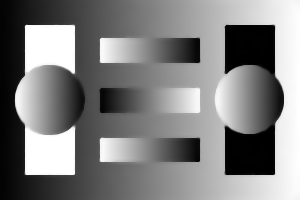}
}
\subfigure[Inf-convolution deblurring, SSIM=\newline0.9466]{
\includegraphics[scale=0.5035]{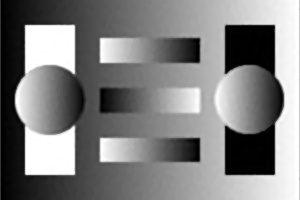}
}
\subfigure[TV-TV$^{2}$ deblurring, $\alpha=0.004$, $\beta=0.0001$, SSIM=0.9739]{
\includegraphics[scale=0.63]{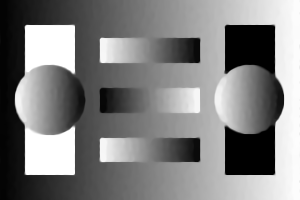}
}
\subfigure[TV-TV$^{2}$ deblurring, $\alpha=0.004$, $\beta=0.0002$, SSIM=0.9710]{
\includegraphics[scale=0.63]{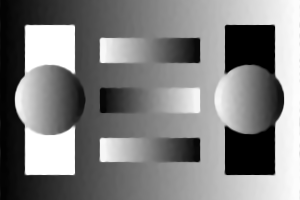}
}
\subfigure[TV$^{2}$ deblurring,  $\beta=0.0012$, SSIM=\newline0.9199]{
\includegraphics[scale=0.63]{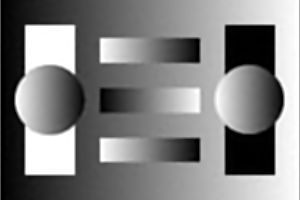}
}
\caption{Deblurring of a blurred (Gaussian kernel of variance $\sigma=2$) and noisy (additive Gaussian noise, variance $10^{-4}$) synthetic image. We chose $\lambda_{1}=100\alpha$, $\lambda_{2}=100\beta$ for these implementations.}
\label{deblurringimages}
\end{figure*}

\begin{figure*}[ht]
\begin{center}
\subfigure[Clean, image, SSIM=1]{
\includegraphics[scale=0.3]{fig17.png}
}
\subfigure[Blurred and noisy image, \newline SSIM=0.8003]{
\includegraphics[scale=0.3]{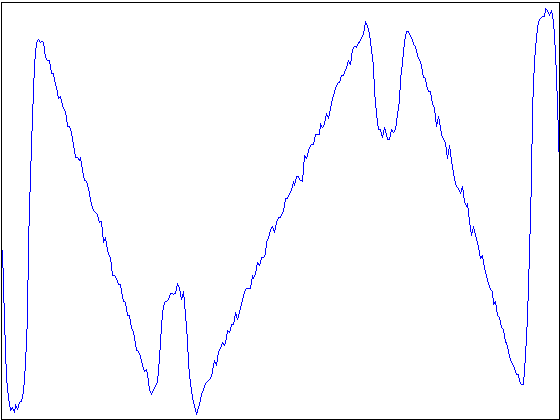}
}
\subfigure[TV deblurring, $\alpha$$=$$0.006$, SSIM=\newline0.9680]{
\includegraphics[scale=0.3]{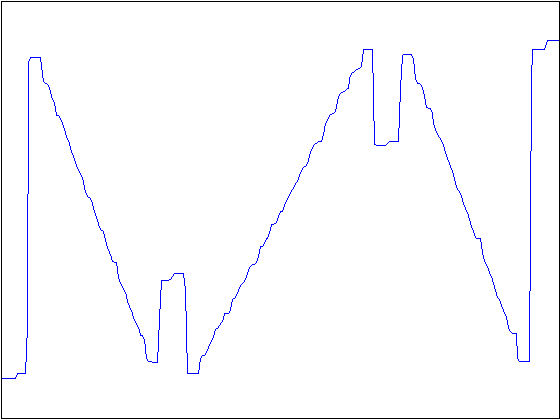}
}
\subfigure[TGV deblurring, SSIM=0.9806 ]{
\includegraphics[scale=0.24]{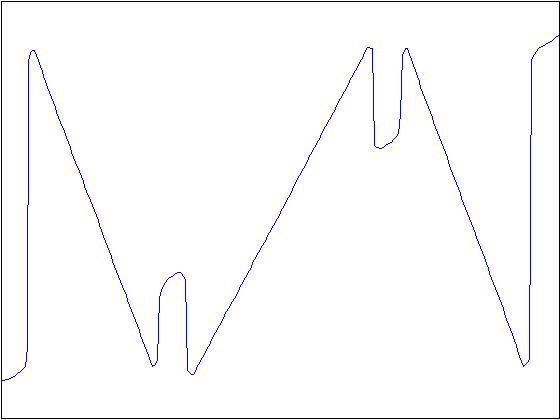}
}
\subfigure[Inf-convolution deblurring, \newline SSIM=0.9466]{
\includegraphics[scale=0.24]{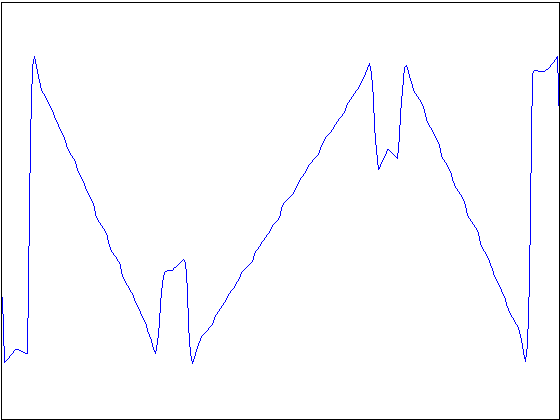}
}
\subfigure[TV-TV$^{2}$ deblurring, $\alpha=0.004$, $\beta=0.0001$, SSIM=0.9739]{
\includegraphics[scale=0.3]{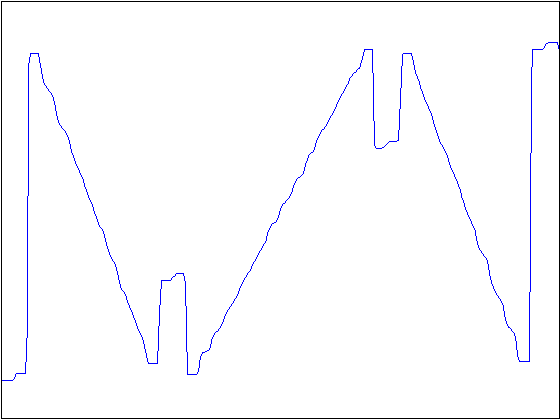}
}
\subfigure[TV-TV$^{2}$ deblurring, $\alpha=0.004$, $\beta=0.0002$, SSIM=0.9710]{
\includegraphics[scale=0.3]{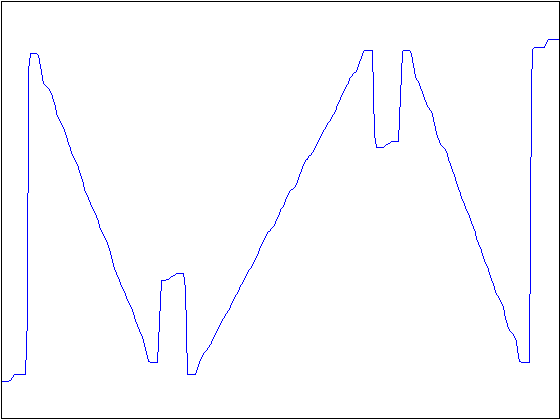}
}
\subfigure[TV$^{2}$ deblurring,  $\beta$$=$$0.0012$, \newline SSIM=0.9199]{
\includegraphics[scale=0.3]{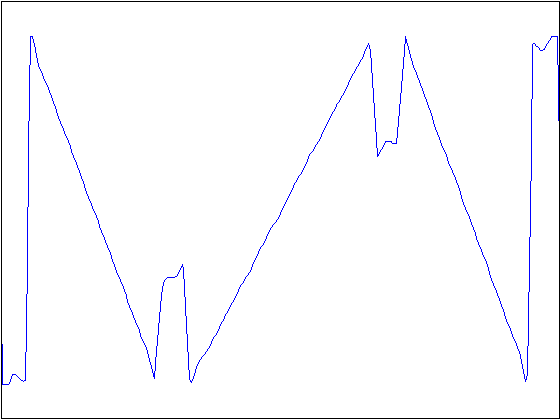}
}
\caption{Corresponding middle row slices of images in Figure \ref{deblurringimages}}
\label{deblurringslices}
\end{center}
\end{figure*}

\begin{figure*}[ht]
\begin{center}
\subfigure[Clean image, SSIM=1]{
\includegraphics[scale=0.45]{fig68.png}
}
\subfigure[Blurred and noisy image,\newline  SSIM=0.7149]{
\includegraphics[scale=0.45]{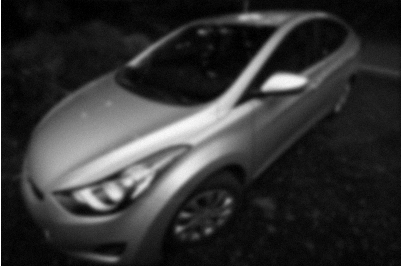}
}
\subfigure[TV deblurring, $\alpha=0.0007$,\newline  SSIM=0.8293]{
\includegraphics[scale=0.45]{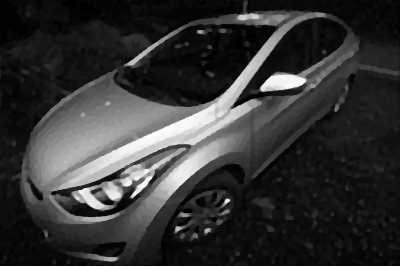}
}

\subfigure[TV-TV$^{2}$ deblurring, $\alpha=0.0005$,\newline $\beta=0.0001$, SSIM=0.8361]{
\includegraphics[scale=0.45]{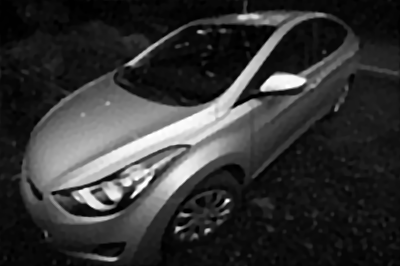}
}
\subfigure[TV-TV$^{2}$ deblurring, $\alpha=0.0005$,\newline $\beta=0.0003$, SSIM=0.8307]{
\includegraphics[scale=0.45]{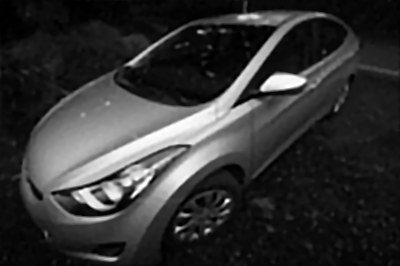}
}
\subfigure[TV-TV$^{2}$ deblurring, $\alpha=0.0005$,\newline $\beta=0.0003$, SSIM=0.8330,\newline Post-processing: GIMP sharpening]{
\includegraphics[scale=0.45]{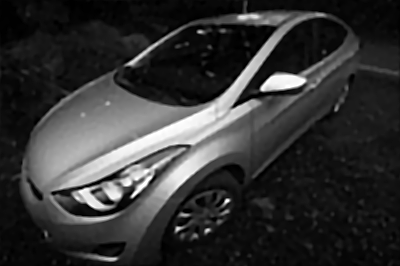}
}

\subfigure[Clean image, SSIM=1]{
\includegraphics[scale=1.2]{fig74.png}
}
\subfigure[Blurred and noisy image, SSIM=0.7149]{
\includegraphics[scale=1.2]{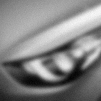}
}
\subfigure[TV deblurring, \newline $\alpha=0.0007$,SSIM=0.8293]{
\includegraphics[scale=1.2]{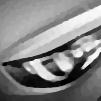}
}

\subfigure[TV-TV$^{2}$ deblurring,\newline $\alpha=0.0005$, $\beta=0.0001$,\newline SSIM=0.8361]{
\includegraphics[scale=1.2]{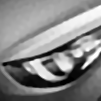}
}
\subfigure[TV-TV$^{2}$ deblurring,\newline $\alpha=0.0005$, $\beta=0.0003$,\newline SSIM=0.8307]{
\includegraphics[scale=1.2]{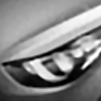}
}
\subfigure[TV-TV$^{2}$ deblurring, $\alpha=0.0005$, $\beta=0.0003$, SSIM=0.8330,\newline Post-processing: GIMP \newline sharpening]{
\includegraphics[scale=1.2]{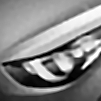}
}
\caption{Deblurring of a blurred (Gaussian kernel of variance $\sigma=2$) and noisy (additive Gaussian noise, variance $10^{-4}$) natural image. We chose $\lambda_{1}=100\alpha$, $\lambda_{2}=100\beta$ for these implementations}
\label{deblurringnatural}
\end{center}
\end{figure*}

In our deblurring implementation $T$ denotes a circular convolution with a discrete approximation of a Gaussian kernel ($\sigma=2$,  size: $11\times 11$ pixels). The blurred image is also  corrupted by additive Gaussian noise of variance $10^{-4}$. Let us note here that  the optimality condition \eqref{optcont} can be solved very fast using fast Fourier transforms. Deblurring results are shown in Figure \ref{deblurringimages} and the corresponding middle row slices in Figure \ref{deblurringslices}. As in the denoising case the introduction of the second order term with a small weight $\beta$ decreases noticeably the staircasing effect, compare Figures \eqref{deblurringimages}(c) and (f). Moreover, we can achieve better visual results if we increase further the value of  $\beta$ without blurring the image significantly, Figure \eqref{deblurringimages}(g). Infimal convolution does not give a satisfactory result here, Figure \ref{deblurringimages}(e). TGV gives again the best qualitative result, Figure \eqref{deblurringimages}(d), but the computation takes about 10 minutes. Even though the time comparison is not completely fair here (the implementation described in \cite{tgvcolour} does not use FFT) it takes a few thousands iterations for TGV to deblur the image satisfactorily, in comparison with a few hundreds for our TV-TV$^{2}$ method.

In Figure \ref{deblurringnatural} we discuss the performance of the TV-TV$^2$ method for deblurring a natural image. The best result for  the TV-TV$^{2}$ method  (SSIM=0.8361) is achieved with $\alpha=0.0005$ and  $\beta=0.0001$, Figure \ref{deblurringnatural}(d). As in the case of denoising, one can increase the value of $\beta$ slightly, eliminating further the staircasing effect, Figures \ref{deblurringnatural}(e) and (k). The additional blur which is a result of the larger $\beta$ can be controlled using a sharpening filter, Figures \ref{deblurringnatural}(f) and (l).

\section{Applications in inpainting}\label{inpainting}
\begin{figure*}[ht]
\begin{center}
\subfigure[Initial image, the gray area denotes the inpainting domain $D$]{
\includegraphics[scale=0.85]{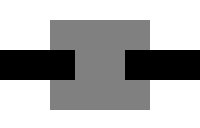}
}
\subfigure[Harmonic inpainting, $\alpha=0.01$]{
\includegraphics[scale=0.85]{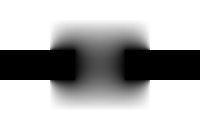}
}
\subfigure[TV inpainting, $\alpha=0.01$]{
\includegraphics[scale=0.85]{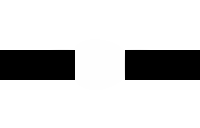}
}
\subfigure[Euler's elastica inpainting]{
\includegraphics[scale=0.68]{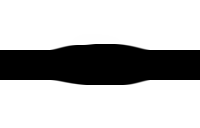}
}
\subfigure[TV-TV$^{2}$ inpainting, $\alpha=0.005$, $\beta=0.005$]{
\includegraphics[scale=0.85]{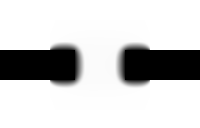}
}
\subfigure[TV$^{2}$ inpainting, $\beta=0.01$]{
\includegraphics[scale=0.85]{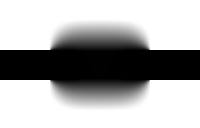}
}
\subfigure[TV$^{2}$ inpainting -- post-processing with shock filter \cite{alvarez1994signal}]{
\includegraphics[scale=0.7]{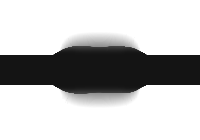}
}

\caption{Comparison of different inpainting methods regarding connectivity across large gaps}
\label{inpainting1}
\end{center}
\end{figure*}

\begin{figure*}[ht]
\begin{center}
\subfigure[Domain 1]{
\includegraphics[scale=0.5]{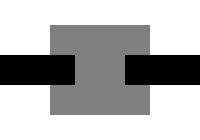}
}
\subfigure[Domain 6]{
\includegraphics[scale=0.5]{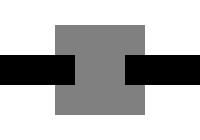}
}
\subfigure[Domain 7]{
\includegraphics[scale=0.5]{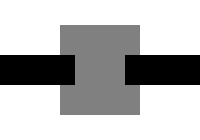}
}
\subfigure[Domain 8]{
\includegraphics[scale=0.5]{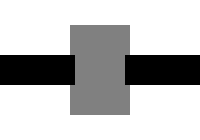}
}
\subfigure[Domain 9]{
\includegraphics[scale=0.5]{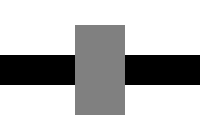}
}
\subfigure[Domain 1,$\,$TV$^{2}$]{
\includegraphics[scale=0.5]{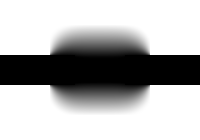}
}
\subfigure[Domain 6,$\,$TV$^{2}$]{
\includegraphics[scale=0.5]{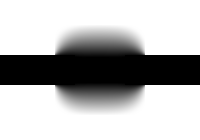}
}
\subfigure[Domain 7,$\,$TV$^{2}$]{
\includegraphics[scale=0.5]{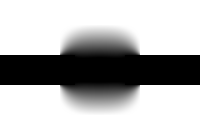}
}
\subfigure[Domain 8,$\,$TV$^{2}$]{
\includegraphics[scale=0.5]{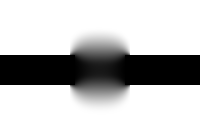}
}
\subfigure[Domain 9,$\,$TV$^{2}$]{
\includegraphics[scale=0.5]{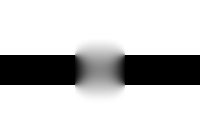}
}
\caption{Different pure TV$^{2}$ inpainting results for different inpainting domains of decreasing width. In all computations we set $\beta=0.001$. }
\label{inpainting2}
\end{center}
\end{figure*}

\begin{figure*}[ht]
\begin{center}
\subfigure[Initial image]{
\includegraphics[scale=0.55]{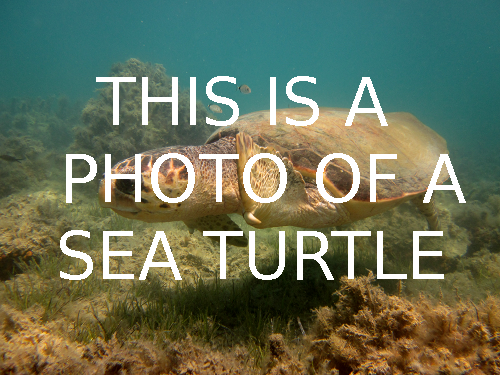}
}
\subfigure[TV inpainting, $\alpha=0.001$]{
\includegraphics[scale=0.55]{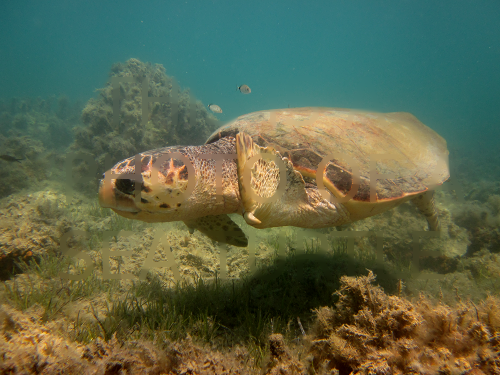}
}
\subfigure[Euler's elastica inpainting]{
\includegraphics[scale=0.44]{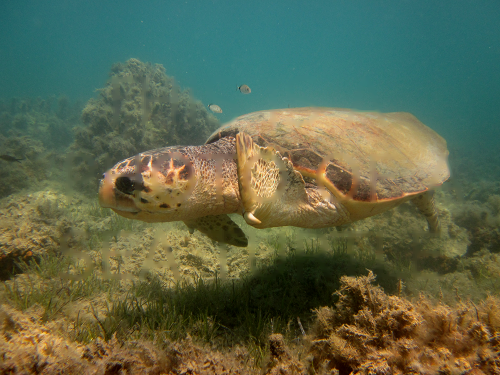}
}
\subfigure[TV$^2$ inpainting, $\beta=0.001$]{
\includegraphics[scale=0.55]{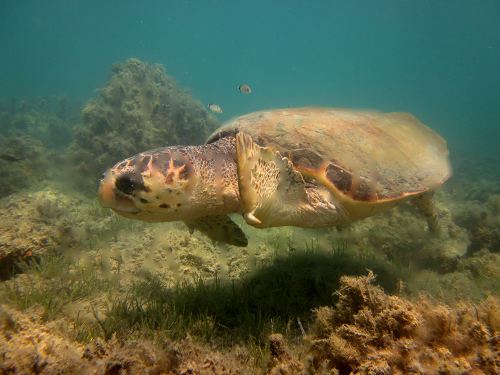}
}
\subfigure[Initial image: detail]{
\includegraphics[scale=1.8]{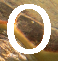}
}
\subfigure[TV inpainting: detail]{
\includegraphics[scale=1.8]{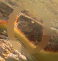}
}
\subfigure[Euler's elastica inpainting: detail]{
\includegraphics[scale=1.8]{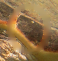}
}
\subfigure[TV$^2$ inpainting: detail]{
\includegraphics[scale=1.8]{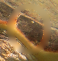}
}
\caption{Removing large font text from a natural image. The TV$^{2}$ result is comparable with the Euler's elastica one.}
\label{inpainting3}
\end{center}
\end{figure*}

Finally, we present examples for the application of our TV-TV$^{2}$ approach to image inpainting. There, the goal is to reconstruct an image inside a missing part using information from the intact part. The missing part is a domain $D\subseteq \Omega$, known as the \emph{inpainting domain}. In this case the operator $T$ applied to an image $u$ gives
\[Tu=\mathcal{X}_{\Omega\setminus D}u,\]
where as before $\mathcal{X}_{\Omega\setminus D}$ is the characteristic function of $\Omega\setminus D$, the intact part of the image domain.
Let us note here that in cases of a small number of inpainting domains and a completely noise free and trustable image outside of the missing part, it is preferable to solve the inpainting problem only inside the holes, see for example \cite{cai2008framelet}. However, here we would like to keep the method flexible, such that it includes the case where there is noise in the known regions as well.

In order to take advantage of the FFT for the optimality condition \eqref{optcont} a different splitting technique to \eqref{minimisationstep2} is required: 
\begin{equation}\label{split2}
\min_{\substack{u\in \mathbb{R}^{n\times m}\\ \tilde{u}\in\mathbb{R}^{n\times m} \\ v \in \left (\mathbb{R}^{n\times m}\right)^{2}  \\ w \in \left (\mathbb{R}^{n\times m}\right)^{3}}  }  \|\mathcal{X}_{\Omega\setminus D}(u-u_{0})\|_{2}^{2}+\alpha \|v\|_{1}+\beta \|w\|_{1}, 
\end{equation}
such that  $u=\tilde{u}$, $v=\nabla \tilde{u}$, $w= \nabla^{2}\tilde{u}$. We refer the reader to \cite{mineipol} for the details. 

In Figure \ref{inpainting1} we compare our method with harmonic and  Euler's elastica inpainting, see \cite{chan2002euler,tai2010fast}. In the case of harmonic inpainting the regulariser is the square  of the $L^{2}$ norm of the gradient $\int_{\Omega} |\nabla u|^{2}~dx$. In the case of Euler's elastica the
regulariser is 
\begin{equation}\label{elastica}
\int_{\Omega} \left (\alpha+\beta\left (\nabla \cdot \frac{\nabla u}{|\nabla u|} \right )^{2}  \right )|\nabla u|~dx.
\end{equation}
The minimisation of the Euler's elastica energy corresponds to the minimisation of the length and curvature of the level lines of the image. Thus, this method is able to connect  large gaps in the inpainting domain, see Figure \ref{inpainting1}(d). However, the term \eqref{elastica} is non-convex and thus difficult to be minimised. In order to implement Euler's elastica inpainting we used the augmented Lagrangian method, proposed in \cite{tai2010fast}. There, the leading computational cost per iteration is due to the solution of one linear PDE and a system of linear PDEs (solved with FFT as well), in comparison to our approach which consists of one linear PDE only. Hence, in Table \ref{times}, we do not give absolute computational times as that would not be fair even more so because we did not optimise the Euler's elastica algorithm with respect to the involved parameters. Moreover, it should be emphasised that the solution of TV-TV$^2$ inpainting amounts to solve a convex problem while Euler's elastica inpainting is a non-convex model.

In Figure \ref{inpainting1} we see that, in contrast to harmonic and TV inpainting, TV$^{2}$ inpainting is able to connect big gaps with the price of a blur, which can be controlled using a shock filter \cite{alvarez1994signal}, see Figure \ref{inpainting1}(g). Notice that one has to choose $\alpha$ small or even 0, in order to make the TV$^{2}$ term  dominant, compare Figures \ref{inpainting1}(e) and (f).
 
 We  also observe that in the TV$^{2}$ case, the ability to connect  large gaps depends on the size and geometry of the inpainting domain, see Figure \ref{inpainting2} and also see \cite{mineipol} for more examples. Deriving sufficient conditions on the size and geometry of the inpainting domain for this connectivity to happen is a matter of future research.
 
 Finally, in Figure \ref{inpainting3} we compare TV, TV$^{2}$ and Euler's elastica, for its application to removing text (of large font) from a natural image.  TV inpainting, Figure \ref{inpainting3}(b) gives unsatisfactory results, by producing piecewise constant results inside the inpainting domain, while the TV$^{2}$ and Euler's elastica results are comparable and seem to be visually closer to the true solution than the TV inpainted image.
 
 \section{Comparison with other higher-order methods}\label{comparison}
In the case of denoising and deblurring we compared our method with TGV, which we consider a state of the art method in the field of higher-order image reconstruction in the variational context. Indeed, in both image reconstruction tasks, TGV gives better qualitative results, in terms of the SSIM index. However the computational time that was needed to obtain the TGV result solved with the primal-dual method is significantly more than the one that is needed to compute the TV-TV$^{2}$ method using split Bregman, see Table \ref{times}. We also show that with  simple and fast post-processing techniques we can obtain results comparable with TGV. For these reasons, we think that the TV-TV$^2$ approach is in particular interesting for applications in which the speed of computation matters. Regarding the comparison with inf-convolution, our method is slightly faster and  results in better reconstructions in deblurring while in denoising the results are comparable.
As far as inpainting is concerned, we compared our method (essentially the pure TV$^{2}$ approach) with the Euler's elastica, a higher-order variational method which is capable of giving very good results, by connecting large gaps in the inpainting domain. However, the regulariser there is non convex, something that could make the minimisation process produce a local minimiser instead of a global one. In TV$^{2}$ the regulariser --  regarded as a convex simplification of the Euler's elastica idea -- it has the ability to connect large gaps and the slight blur that is produced can be reduced by using a shock filter see for example \cite{alvarez1994signal,gilboa2004image} and Figure \ref{inpainting1}(g). Moreover as we pointed out  in the previous sections, our approach is computationally less expensive compared to TGV for image denoising and deblurring and Euler elastica for image inpainting.

\begin{table*}
\begin{center}
\resizebox{18cm}{!}{
\begin{tabular}{| c | c | c | c | c |}
\hline
\multicolumn{5}{|c|}{\large{\textbf{Denoising (B\&W) -- Image size: 200$\times$300}}}\\
\hline \hline
 & No of iterations for GT & No of iterations for $\|u_{k}-GT\|_{2}/\|GT\|_{2}\le 10^{-3}$ & CPU time (secs)& time per iteration (secs) \\
   \hline
   \textbf{TV}          					&  2000 & 136 & 2.86 &0.0210\\
   \hline
   \textbf{TV}$^{\mathbf{2}}$ 				& 2000  &107 & 3.62 &0.0338\\
   \hline
   \textbf{TV-}\textbf{TV}$^{\mathbf{2}}$	& 2000 & 86 &4.05 & 0.0471\\
   \hline
   \textbf{Inf.-Conv.}					& 2000 & 58 &3.33 & 0.0574\\
   \hline
   \textbf{TGV}							& 2000 & 1297 & 36.25 & 0.0279\\
\hline \hline
\multicolumn{5}{|c|}{\large{\textbf{Deblurring (B\&W) -- Image size: 200$\times$300}}}\\
\hline \hline
 & No of iterations for GT & No of iterations for $\|u_{k}-GT\|_{2}/\|GT\|_{2}\le 10^{-2}$ & CPU time (secs)& time per iteration (secs) \\
   \hline
   \textbf{TV} 							& 1000 & 478 & 10.72 &0.0257\\
   \hline
   \textbf{TV}$^{\mathbf{2}}$				& 1000& 108 & 3.64 & 0.0337\\
   \hline
   \textbf{TV-}\textbf{TV}$^{\mathbf{2}}$	& 1000 & 517& 25.47 & 0.0493\\
   \hline
   \textbf{Inf.-Conv.}					& 1000 &108 & 7.47 & 0.0692\\
   \hline
   \textbf{TGV} & \multicolumn{3}{|c|}{CPU time more than 10 minutes -- see relevant comment in Section \ref{deblurring}}  & 1.22   \\
   \hline \hline
\multicolumn{5}{|c|}{\large{\textbf{Inpainting (Colour) -- Image size: 375$\times$500}}}\\
\hline \hline
&   \multicolumn{2}{|c|}{No of iterations for $\|u_{k}-u_{k-1}\|_{2}/\|u_{k-1}\|_{2}\le 8\cdot 10^{-3}$}       & CPU time (secs)  & time per iteration (secs)\\
\hline
\textbf{TV}                             &  \multicolumn{2}{|c|}{88}  &  26.67  & 0.3031 \\	
\hline
\textbf{TV}$^{\mathbf{2}}$	   &  \multicolumn{2}{|c|}{103}  & 46.34  & 0.4499 \\			\hline	
 \textbf{E.e.} & \multicolumn{4}{|c|}{CPU time more than 10 minutes -- see relevant comment in Section \ref{inpainting}}    \\
 \hline			
\end{tabular}
}
\vspace{0.2 cm}
\caption{Computational times for the examples of Figures \ref{denoisingimages}, \ref{deblurringimages} and \ref{inpainting3}. For the denoising and deblurring examples we  computed a ground true solution (GT) for every method by taking a large number of iterations and record the number of iterations and CPU time it takes for the relative residual of the iterates and the ground true solution to fall below a certain threshold. For the inpainting, we give the number of iterations and CPU time it took to compute the solutions shown in Figure \ref{inpainting3}, where we chose as  a stopping criterium a small relative residual of the iterates. The TGV examples were computed using $\sigma=\tau=0.25$ in the primal-dual method described in \cite{tgvcolour}. The implementation was done using MATLAB (2011) in a Macbook 10.7.3, 2.4 GHz Intel Core 2 Duo and 2 GB of memory}
\label{times}
\end{center}
\end{table*}

\section{Conclusion}
We formulate a second order variational problem in the space of functions of bounded Hessian in the context of convex functions of measures. We prove existence and uniqueness of minimisers using a relaxation technique as well as stability. We propose the use of the split Bregman method for the numerical solution of the analogue discretised problem. The application of the split Bregman method to our model is quite robust and is converging after a few iterations. We perform numerical experiments by denoising images that have been corrupted by Gaussian  noise, deblurring  images that have been convoluted with Gaussian kernels, as well as in image inpainting.

In the case of denoising and deblurring, the introduction of the second order term leads to a significant reduction of the staircasing effect resulting in piecewise smooth images rather than piecewise constant images when using the ROF model. The superiority of an approach that combines first and second order regularisation rather than first order regularisation only, is confirmed quantitatively by the SSIM index. In the case of inpainting the higher-order method is able to connect edges along large gaps, a task that TV-inpainting is incapable of solving.

In summary, our approach is a simple and convex higher-order extension of total variation regularisation that improves the latter by reducing the staircasing effect in image denoising and deblurring, and by more faithfully respecting the good continuation principle in image inpainting. It can compete with other higher-order methods of its kind by giving almost comparable qualitative results while computing them in a fraction of time.

As fas as future work is concerned, a (not necessarily rigorous) rule for selecting the parameters $\alpha$ and $\beta$ would be useful. Investigating the links with inf-convolution and spatially dependent regularisation on our model is also of interest.  Moreover, the relation between the continuum and the discretised model could be investigated through $\Gamma$-convergence arguments, see \cite{dalmasogamma} and \cite{Braidesgamma}. Finally the characterisation of subgradients of this approach and the analysis of solutions of the corresponding PDE flows for different choices of functions f and g promises to give more insight into the qualitative properties of this regularisation procedure. The characterisation of subgradients will also give more insight to properties of exact solutions of the minimisation of \eqref{functional} concerning the avoidance of the staircasing effect.

\subsection*{Acknowledgements}The authors acknowledge the financial support provided by the Cambridge Centre for Analysis (CCA), the Royal Society International Exchanges Award IE110314 for the project ``High-order Compressed Sensing for Medical Imaging", the EPSRC / Isaac Newton Trust Small Grant ``Non-smooth geometric reconstruction for high resolution MRI imaging of fluid transport in bed reactors'' and the EPSRC first grant Nr. EP/J009539/1 ``Sparse \& Higher-order Image Restoration''. Further, this publication is based on work supported by Award No. KUK-I1-007-43 , made by King Abdullah University of Science and Technology (KAUST). We thank Clarice Poon for providing us with the Euler's elastica code. Finally, we would like to thank the referees for their very useful comments and suggestions that improved the presentation of the paper.

\appendix
\section{Some useful theorems}\label{appendixA}
\newtheorem{firstth}{Proposition}[section]
\begin{firstth}\label{convexmeasure}
Suppose that $g:\mathbb{R}^{m}\to\mathbb{R}$ is a continuous function, positively homogeneous of degree $1$ and let $\mu\in[\mathcal{M}(\Omega)]^{m}$. Then for every positive measure Radon measure $\nu$ such that $\mu$ is absolulely continuous with respect to $\nu$, we have
\[g(\mu)=g\left (\frac{\mu}{\nu} \right )\nu.\] 
Moreover, if $g$ is a convex function, then $g:[\mathcal{M}(\Omega)]^{m}\to \mathcal{M}(\Omega)$ is a convex function as well.
\end{firstth}
\begin{proof}
Since $\mu\ll\nu$, we have that $|\mu|\ll \nu$. Using the fact that $g$ is positively homogeneous and the fact that $|\mu|/\nu$ is a positive function, we get
\begin{equation*}
g(\mu)=g\left (\frac{\mu}{|\mu|} \right )|\mu|
	   =g\left (\frac{\mu}{|\mu|} \right) \frac{|\mu|}{\nu}\nu=g\left (\frac{\mu}{|\mu|} \frac{|\mu|}{\nu} \right)\nu=g\left (\frac{\mu}{\nu} \right)\nu.
\end{equation*}
Assuming that $g$ is convex and using the first part of the proposition we get for $0\le\lambda\le 1$, $\mu$, $\nu\in [\mathcal{M}(\Omega)]^{m}$:
\begin{eqnarray*}
g(\lambda \mu+(1-\lambda)\nu)&=&g\left( \frac{\lambda \mu+(1-\lambda)\nu}{|\lambda \mu+(1-\lambda)\nu|} \right) |\lambda \mu+(1-\lambda)\nu|\\
&=& g\left( \frac{\lambda \mu+(1-\lambda)\nu}{|\mu|+|\nu|} \right) (|\mu|+|\nu|)\\
&=&g\left(\lambda \frac{\mu}{|\mu|+|\nu|}+(1-\lambda)\frac{\nu}{|\mu|+|\nu|} \right) (|\mu|+|\nu|)\\
&\le&\lambda g\left( \frac{\mu}{|\mu|+|\nu|} \right) (|\mu|+|\nu|)+(1-\lambda)g\left (\frac{\nu}{|\mu|+|\nu|}\right) (|\mu|+|\nu|)\\
&=&\lambda g\left( \frac{\mu}{|\mu|} \right) |\mu|+(1-\lambda)g\left (\frac{\nu}{|\nu|}\right) |\nu|\\
&=&\lambda g(\mu)+(1-\lambda) g(\nu).\\
\end{eqnarray*}
\end{proof}

The following theorem which is a special case of a theorem that was proved in \cite{buttazzo1991functionals} and can be also found in \cite{AmbrosioBV} establishes the lower semicontinuity of convex functionals of measures with respect to the weak$^{\ast}$ convergence.

\newtheorem{firstth2}[firstth]{Theorem}
\begin{firstth2}[Buttazzo-Freddi, 1991]\label{BF}
Let $\Omega$ be an open subset of $\mathbb{R}^{n}$, $\nu$, $(\nu_{k})_{k\in \mathbb{N}}$ be $\mathbb{R}^{m}$-valued finite Radon measures and $\mu$, $(\mu_{k})_{k\in\mathbb{N}}$ be positive Radon measures in $\Omega$. Let $g:\mathbb{R}^{m}\to\mathbb{R}$ be a convex  function and suppose that $\nu_{k}\to\nu$ and $\mu_{k}\to\mu$ weakly$^{\ast}$ in $\Omega$. Consider the Lebesgue decompositions $\nu=(\nu/\mu)\mu+\nu^{s}$, $\nu_{k}=(\nu_{k}/\mu_{k})\mu_{k}+\nu_{k}^{s}$, $k\in \mathbb{N}$. Then 
\[\int_{\Omega}g\left( \frac{\nu}{\mu}(x)\right)d\mu(x)+\int_{\Omega}g_{\infty}\left(\frac{\nu^{s}}{|\nu^{s}|}(x) \right)d|\nu^{s}|(x)\le \liminf_{k\to\infty}\int_{\Omega}g\left( \frac{\nu_{k}}{\mu_{k}}(x)\right)d\mu_{k}(x)+\int_{\Omega}g_{\infty}\left(\frac{\nu_{k}^{s}}{|\nu_{k}^{s}|}(x) \right)d|\nu_{k}^{s}|(x).\]
In particular, if $\mu=\mu_{k}=\mathcal{L}^{n}$ for all $k\in \mathbb{N}$ then according to the definition \eqref{defconv} the above inequality can be written as follows:
\[g(\nu)(\Omega)\le \liminf_{k\to\infty}g(\nu_{k})(\Omega).\]
\end{firstth2}
The following theorem is a special case of Theorem $2.3$ in \cite{Dem}.
\newtheorem{dem84}[firstth]{Theorem}
\begin{dem84}[Demengel-Temam, 1984]\label{demtem}
Suppose that $\Omega\subseteq \mathbb{R}^{n}$ is open, with Lipschitz boundary and let  $g$ be a convex function from $\mathbb{R}^{n\times n}$ to $\mathbb{R}$ with at most linear growth at infinity. Then for every $u\in BH(\Omega)$ 
there exists a sequence $(u_{k})_{k\in \mathbb{N}}\subseteq C^{\infty}(\Omega)\cap W^{2,1}(\Omega)$ such that
 \[\|u_{k}-u\|_{L^{1}(\Omega)}\to 0, \quad |D^{2}u_{k}|(\Omega)\to |D^{2} u|(\Omega),\]\[  g(D^{2}u_{k})(\Omega)\to g(D^{2}u)(\Omega),\;\;\text{ as }\;\;k\to\infty. \]
\end{dem84}

\newtheorem{Kron}[firstth]{Lemma}
\begin{Kron}[Kronecker's lemma]\label{kronecker}
Suppose that $(a_{n})_{n\in\mathbb{N}}$ and $(b_{n})_{n\in\mathbb{N}}$ are two  sequences of real numbers such that \\$\sum_{n=1}^{\infty}a_{n}<\infty$ and $0<b_{1}\le b_{2}\le\ldots$ with $b_{n}\to\infty$. Then 
\[\frac{1}{b_{n}}\sum_{k=1}^{n}b_{k}a_{k}\to 0,\quad \text{as } n\to\infty.\]
In  particular, if $(c_{n})_{n\in\mathbb{N}}$ is a decreasing positive real sequence such that $\sum_{n=1}^{\infty}c_{n}^{2}<\infty$, then
\[c_{n}\sum_{k=1}^{n}c_{k}\to 0,\quad \text{as } n\to\infty.\]
\end{Kron}

\bibliographystyle{abbrv}

\bibliography{kostasbib}

\end{document}